\documentclass[12pt,reqno]{amsart}
\usepackage{amsmath,amsfonts,amssymb,amscd,amsthm,amsbsy}
\usepackage[dvips]{graphicx}
\usepackage{comment,color}
\numberwithin{equation}{section}
\newtheorem{theorem}{Theorem}
\newtheorem{meta-thm}[theorem]{Meta-Theorem}
\newtheorem{lemma}[theorem]{Lemma}
\newtheorem{cor}[theorem]{Corollary}
\newtheorem{proposition}[theorem]{Proposition}

\newtheorem{remark}[theorem]{Remark}
\newtheorem{definition}[theorem]{Definition}

\newtheorem{conjecture}[theorem]{Conjecture}

\newcommand\beq[1]{ \begin{equation}\label{#1} }
\newcommand{\eeq}{ \end{equation} }

\newcommand\beqa[1]{ \begin{eqnarray} \label{#1}}

\newcommand{\eeqa}{ \end{eqnarray} }

\newcommand{\beqano}{ \begin{eqnarray*} }

\newcommand{\eeqano}{ \end{eqnarray*} }

\newcommand\equ[1]{{\rm (\ref{#1})}}
\newcommand{\av}[1]{\langle #1 \rangle}

\def\ep{\varepsilon}

\def\dist{\operatorname{dist}}

\def\Im{\operatorname{Im}}

\def\A{{\mathcal A}}
\def\C{{\mathcal C}}
\def\D{{\mathcal D}}
\def\O{{\mathcal O}}
\def\P{{\mathcal P}}
\def\Lop{{\mathcal L}}
\def\M{{\mathcal M}}
\def\complex{{\mathbb C}}
\def\integer{{\mathbb Z}}
\def\nat{{\mathbb N}}

\def\real{{\mathbb R}}
\def\torus{{\mathbb T}}

\def\Tau{ \mathcal{T}}

\def\B{{\mathcal B}}
\def\F{{\mathcal F}}
\def\G{{\mathcal G}}
\def\T{{\mathcal T}}

\begin{document}

\title[Response solutions for dissipative wave equations]
{Response solutions for quasi-periodically forced, dissipative wave equations}

\author[R. Calleja]{Renato C.  Calleja}
\address{Department of Mathematics and Mechanics, IIMAS, National
  Autonomous University of Mexico (UNAM), Apdo. Postal 20-126,
  C.P. 01000, Mexico D.F., Mexico}
\email{calleja@mym.iimas.unam.mx}

\author[A. Celletti]{Alessandra Celletti}
\address{
Department of Mathematics, University of Roma Tor Vergata, Via della Ricerca Scientifica 1,
00133 Rome, Italy}
\email{celletti@mat.uniroma2.it}

\author[L. Corsi]{Livia Corsi}
\address{
Department of Mathematics and Statistics, McMaster University, Hamilton ON, L8S 4K1, Canada}
\email{lcorsi@math.mcmaster.ca}

\author[R. de la Llave]{Rafael de la Llave}
\address{
School of Mathematics,
Georgia Institute of Technology,
686 Cherry St., Atlanta GA. 30332-0160, USA}
\email{rafael.delallave@math.gatech.edu}

\thanks{R.C. was partially supported by NSF grant DMS-1162544 and CONACYT grant 133036.
A.C. was partially supported by PRIN-MIUR 2010JJ4KPA$\_$009 and GNFM/INdAM.
L.C. was partially supported by the ERC project
``Hamiltonian PDEs and small divisor problems: a dynamical systems approach''.
R.L. was partially supported by NSF grant DMS-1162544. }

\baselineskip=18pt              


\begin{abstract}
We consider several models of nonlinear wave equations subject
to very strong damping
and quasi-periodic external forcing. This is a singular perturbation, since the damping is not the highest
order term. We study the existence of response solutions
(i.e., quasi-periodic solutions with the same frequency as the forcing).

Under very general non-resonance conditions on the frequency, we show
the existence of asymptotic expansions of the response solution;
moreover, we prove that the response solution indeed exists and depends
analytically on $\varepsilon$ (where $\varepsilon$ is the inverse of the
coefficient multiplying the damping) for $\varepsilon$ in a complex domain, which in some cases
includes disks tangent to the imaginary axis at the origin.
In other models, we prove analyticity in cones of aperture $\pi/2$ and we conjecture it is optimal.
These results
have consequences for the asymptotic expansions of the response
solutions considered in the literature.
The proof of our results relies on reformulating the
problem as a fixed point problem,
constructing an approximate solution and studying the properties of
iterations that converge to the solutions of the
fixed point problem.
\end{abstract}


\maketitle

\tableofcontents

\section{Introduction}
In recent times there has been extensive interest in strongly damped
systems, namely systems in which the  term describing the damping
contains a factor $\varepsilon^{-1}$
(where $\varepsilon$ is a small parameter), and subject to external forcing.
Since the damping is not the term which corresponds to the time-derivative
of highest order, this is a  singular perturbation in $\varepsilon$.
We are interested in finding \emph{response solutions}, i.e. solutions
which have the same frequency as the forcing term.

A first try to understand these problems is to use perturbation
theory in $\varepsilon$ and obtain formal
series in powers of $\varepsilon$. Nevertheless, since the perturbation is
singular, one does not expect that the resulting formal series
is convergent
and one needs to use re-summation techniques to obtain
that there is an analytic solution defined in an open complex domain
which does not include $\varepsilon = 0$, but has it on the boundary.
This approach has been used for ODE's in
\cite{GentileBD05, GentileBD06, Gentile10a, Gentile10b}. Different arguments
for other singular perturbation problems can be found in
\cite{Balser}.

In \cite{CallejaCL13} one can find an alternative approach for singular
problems in ODE's, which inspired our treatment for PDE's. One considers the perturbative expansion to low orders
and obtains a reasonably good approximate solution  in a neighborhood of $\varepsilon=0$
(i.e., an expression that solves the equation up to a small error).
Then, starting from the approximate solution, one switches to another perturbative method
(a contraction mapping argument) to prove the existence of a true solution.
Since the problem is analytic in $\varepsilon$ for $\varepsilon$
ranging in a complex domain, one obtains analytic dependence in $\varepsilon$
of the solution
for $\varepsilon$ in a certain domain which does not include
any  ball
centered at zero. Indeed, we find that there are arbitrarily small
values of $\varepsilon$ for which the map is not a contraction and
the method of proof breaks down. We conjecture that this is a real
effect and not just a shortcoming of the method.

To motivate the procedure adopted in \cite{CallejaCL13}, we argue heuristically
that  since $\varepsilon = 0$  is the most singular value of $\varepsilon$, one
attempts to do as little work as possible based on it.
One tries to implement a perturbation theory on
small but non-zero values of $\varepsilon$; as soon as one gets even a flimsy foothold on non-zero values of $\varepsilon$ one switches to
another perturbation method that is not affected by singularities
(even if it contains some large terms, they can be beaten by
pairing them with small ones).
This procedure is somewhat reminiscent of some
works in  celestial mechanics, notably Hill's
theory of the Moon (\cite{Hill}, \cite[Vol 2]{Poincarefrench}), in which one uses a perturbation theory
from an intermediate model which is controlled in turn by another
perturbative argument.

As it  happens often in perturbative expansions,
the way one deals with the first order term is different from the subsequent ones:
this is even more evident in situations like the present one, since we are dealing with
singular perturbations.
In \cite{CallejaCL13} the first term of the expansion,
corresponding to $\varepsilon = 0$, was obtained by means of an implicit function theorem,
but the subsequent steps were all similar and they involved
the same hypotheses. In this paper, the difference between
the zeroth order term and the higher order ones  is even more dramatic.
The term in the expansion corresponding to
$\varepsilon = 0 $ is very different from the others and in principle can be dealt
with a variety of methods, including implicit function theorems
(at least for certain cases, as we do for the model described by \equ{largestiff} below -- see Section~\ref{sec:othersB}) or using variational
methods (as we can do for the model described by \equ{dissipativewave} below), depending on the model we are studying.
As we will see, when we apply variational
methods, we may get even infinitely many solutions of the order $0$
equation. Each of them will lead to a family of solutions, which is analytic in $\varepsilon$.

Hence, the procedure adopted in the present work has two steps, with the first step having two
substeps.
\begin{enumerate}
\item[a)]
Obtaining an approximate solution to high order, and precisely:
\begin{enumerate}
\item[a1)]
obtaining the order zero solution;
\item[a2)]
obtaining high order approximations.
\end{enumerate}

\item[b)]
Polishing off the approximate solutions to obtain true solutions.
\end{enumerate}

Each  of these steps
has its own methodology (indeed, step a1) will
be accomplished by means of several  different methodologies depending on the model) and requires
different conditions on the frequency as well as different
non-degeneracy assumptions. Hence, the conditions required in
the main theorem are obtained by joining together
the conditions of all the steps.

Nevertheless, the final assumptions are very weak.
For example, the non-resonance conditions needed to carry out the whole problem
are weaker than the Brjuno condition and they allow exponentially growing small divisors.

The strategy above is widely applicable. In this paper
we  decided to document its breadth by applying it to
4 different models in the literature with several variations, such as
different boundary conditions. We call these models A, A', B, B' (more details will be
given in Section~\ref{sec:models}).
On the other hand, we
have not optimized the hypotheses: It seems clear that one could
obtain slightly sharper domains of analyticity, better regularity
conditions, less assumptions on the domain, etc.
We conjecture (and present arguments in favor) that the domains obtained are
essentially optimal (see Section~\ref{sec:optimal}).

The main result for step a) is Theorem~\ref{thm:allorders};
the main result for step b) is Theorem~\ref{convergence}
and the final result is Theorem~\ref{thm:main}.

The step a1) is the solution of a functional equation.
The step a2) is a Lindstedt procedure, which entails very mild conditions on the small divisors
and requires very weak non-resonance conditions on the frequency.
In this way one produces polynomials in $\varepsilon$ which solve
the equation up to some (high) power of $\varepsilon$.
Under a bit stronger conditions on the small divisors, the Lindstedt procedure provides the existence of a
formal solution up to all orders in $\varepsilon$ (see Theorem~\ref{thm:allorders}).
As it turns out, the solutions will be unique once we fix the solution of order $0$; however, as already pointed out,
this solution to order $0$ can be very non-unique.

The step b)
is based on a contraction
mapping principle. Hence no small divisors are involved but, on the
other hand, we need
to consider $\varepsilon$ in an appropriate complex domain to carry out the argument.  We also note that step b) also works in cases where the spectrum of
the operators driving the evolution is not discrete. Unfortunately, we
do not know good conditions that ensure that one can perform step a)
when the spectrum is not discrete. If, by any chance,
one is dealing with a particular problem having a continuous spectrum and step a) can be performed,
then step b) can be performed too and one can obtain the result.

The final result is that the response solution is an analytic function of $\varepsilon$ defined in a domain
(selected in step b) ) which does not include
zero, even if it might include circles with real centers and tangent
to the imaginary axis. Hence, the method does not guarantee that the
Lindstedt series (the formal power expansion) converges, because
the analyticity domain established does not contain any circle centered at the origin.
Indeed, in \cite{CallejaCL13} one can
find arguments that suggest that the Lindstedt series does not converge in
general, even in the case of ODE's. Here, we also present similar arguments
in Section~\ref{sec:optimal}.

Nevertheless, the domain of analyticity established here for models A, A' (describing dissipative wave equations)
is large  enough, so that the application of the
 Nevanlinna-Sokal theory (\cite{Nevanlinna,Sokal,Hardy}) on
asymptotic expansions applies. As a consequence, the response solutions
constructed here
have an asymptotic expansion and these functions can be reconstructed from
their asymptotic expansions
by re-summation. Notice that this procedure is very different from
establishing the existence of the solution by re-summing the series. Of course,
since the problem is nonlinear, re-summing the series is not enough and one
needs other arguments to show that the re-summation
solves the equation (\cite{Hardy}, see also \cite{Balser,BalserLS,GentileBD05, GentileBD06, Gentile10a, Gentile10b}).

In some models such as models B, B' (describing large stiffness equations),
we obtain  domains of analyticity which are cones containing the imaginary axis and have an aperture of $\pi/2$.
We conjecture that these domains are essentially optimal (see Section~\ref{sec:optimal}).
We will show that the functions we construct have the same asymptotic
expansions as the formal power series. On
the other hand we note that in domains of this kind
it is not clear that the response solution can be obtained
by re-summing the asymptotic expansions: indeed in these domains
there are non-trivial functions whose asymptotic expansion vanishes,
so that the expansion is not unique
(e.g., the Cauchy
example $\exp( - \varepsilon^{-2})$ which has an asymptotic expansion vanishing in domains of
aperture $\pi/2$). As a consequence,
it could well happen that for these models the response solutions
lead to exponentially small phenomena. Notice that model B
is an infinite dimensional analogue of fast oscillators for
which exponentially small phenomena have been established
(see \cite{Seara}).

In \cite{CallejaCL13} the problem considered is the
\emph{varactor} equation, which is a single ODE. Even if the
step b) in \cite{CallejaCL13} was just an
elementary one (based on contraction arguments),
the results obtained in \cite{CallejaCL13}
improved the existence domains and weakened
the  non-resonance conditions that have been imposed in the
previous literature.  It seems plausible that using a more efficient
fixed point argument
(e.g., a KAM theory) or higher order
perturbations in step a) would improve the results.
The analyticity domain has later been extended for ODE's in \cite{CFG1}
(where a  domain of analyticity  tangent more than quadratically to the origin was established)
and the non-degeneracy assumption on the non-linearity has been
relaxed (for real $\varepsilon$) in \cite{CFG2}.
As further references, we mention also \cite{rab1,rab2,cr}, where
the periodic case with real small damping has been considered.

\subsection{Description of the main results}
The goal of this paper is to extend the method of \cite{CallejaCL13}
to some PDE's. The method is very flexible and we will present results for
four different models considered in
the literature, each with three different types of boundary conditions
(see Section~\ref{sec:models}). It is clear that there are many more
models that could have been considered by the method. Of course
the main difficulty of the extension to PDE's is
that the operators are unbounded. Hence the reformulation of the problem as
a fixed point problem requires some more thought, even to get  a
viable formulation. For example, we need to ensure that the operator
maps some space into itself and that the space satisfies suitable properties
(such as Banach algebra properties).
So, considerable effort goes into the choice of spaces
as it happened in the classical study of elliptic problems
(see Section~\ref{sec:spaces}).

The models we consider in this paper have the form
$$
\partial_{tt}u(t,x) + \frac{1}{\varepsilon}{\rm Friction}- \Delta_x u(t,x)
+ h(u(t,x) , x) = f(\omega t, x)
$$
for models A, A' below, or
$$
\varepsilon^2 \partial_{tt}u(t,x) + \partial_{t}u(t,x)- \Delta_x u(t,x)
+ {\mbox{ non-linearity}} = f(\omega t, x)\ ,
$$
for models B, B' below. The equations will be supplemented with the boundary conditions.

In all  models, given a domain $\D$ (as specified in Section~\ref{sec:PDEs})
and denoting by $\overline{\D}$ its topological closure,
$u:\real\times \overline{\D}\rightarrow\real$ is the unknown.
We will require that the following data of the problem are fixed:
\begin{itemize}
\item
The boundary conditions;
\item
$h:\real\times \overline{\D}\rightarrow\real$ to which we refer as the non-linearity;
\item
 $f:\torus^d\times \overline{\D}\rightarrow\real$ (with $\torus^d\equiv (\real/\integer)^d$) to which we refer as the
forcing;
\item
$\omega\in\real^d$, which denotes the frequency of the forcing.
We assume without loss of generality that $\omega$ has rationally independent components,
namely that: $\omega \cdot k \ne 0$ for all $k \in \mathbb{Z}^d \setminus \{0\}$.
\end{itemize}

Of course we assume that the forcing and the non-linearity are such that
the boundary conditions are maintained. We will also need:
\begin{itemize}
\item
Quantitative estimates on the size
of $|\omega \cdot k|^{-1}$ as a function of $|k|$ (which will turn out to be  weaker than the Diophantine or
Bryuno conditions);
\item
A non-degeneracy
condition on the non-linearity.
\end{itemize}
Then we shall prove the following ``meta"-result.

\begin{meta-thm}
Under the above requirements there exists a response solution for the models of the form
above, defined for $\varepsilon$ in an appropriate complex domain: the specific form of such domain
depends on the model considered as well as on the boundary conditions.
\end{meta-thm}

The precise statement of the result requires the introduction of the spaces, the domains
and a precise formulation of the regularity condition that we will give later on, see Sections~\ref{sec:res1}-\ref{sec:main}.
The existence of the solutions of the zeroth order term is discussed in Section~\ref{sec:c0} and Appendix~\ref{sec:c0A}.
The proof for the case of dissipative wave equations is provided in Section~\ref{sec:proof}, while
the modifications of the proof for the other models are given in Section~\ref{sec:others}.
Some arguments supporting that the domains are almost optimal are given in Section~\ref{sec:optimal}.

\section{Models considered and some preliminary assumptions}\label{sec:models}

In this section we present the models we intend to study and we state the required non-degeneracy assumptions.

In what follows we will assume that $\Delta_x$ is a self-adjoint
elliptic  operator of second order; in the physical applications we have in mind
it is the Laplace-Beltrami operator. We will not necessarily
assume that $\Delta_x$ is a constant coefficient operator.

\subsection{PDE's considered}\label{sec:PDEs}
We will consider  PDE's for which the space variables
range in the topological closure of a domain $\D$ and we will look for solutions
 quasi-periodic in time.
The domain $\D$ can be:\\

{\bf D1)} a compact manifold without boundary, for example $\torus^\ell$
(we will refer to this as the periodic case),\\

{\bf D2)} an open, bounded, connected subset of $\real^\ell$ with a $C^\infty$ boundary.
In this case, we will supplement the solutions with either Dirichlet or Neumann boundary conditions.\\

Therefore, we will consider the following (standard) boundary conditions, each one leading to
a different functional setting, which we will specify below:

{\bf D) Dirichlet boundary conditions},

{\bf N) Neumann boundary conditions},

{\bf P) Periodic boundary conditions}.\\

Following the usual practice, we interpret the boundary
conditions as describing a space of solutions: the operator
$\Delta_x$  acts on this space and of course the spectral properties of $\Delta_x$ depend on the
space too. Of course, in order to specify the function space we
 also need to specify a norm. Our treatment will be for spaces of
functions which are
analytic in $t$ and differentiable in $x$.

We will consider four different PDE's: models A, A', B, B' below. Each of them may have entirely different boundary conditions
(Periodic, Dirichlet and Neumann).

{\bf A) The dissipative wave model: }
The first model is a direct analogue of the varactor equation studied, e.g., in
\cite{CallejaCL13, CGV, GentileBD05, GentileBD06, Gentile10a, Gentile10b,CFG2,CFG1};
the model is obtained from the wave equation by adding a singular friction proportional to the velocity:
\beq{dissipativewave}
\partial_{tt} u(t,x)+{1\over\varepsilon} \partial _t u(t,x)-\Delta_x u(t,x) + h(u(t,x),x)=f(\omega t,x)\ .
\eeq

{\bf A') The frequency over-damped model:} We modify the friction of model A, as described by the
following equation
\beq{frequencydamped}
 \partial_{tt} u(t,x)+\frac{1}{\varepsilon} \partial _t \Delta_x u(t,x)
-\Delta_x u(t,x) + h(u(t,x),x)=f(\omega t,x)\ .
\eeq
In this model, which has been studied for instance in \cite{PSM},
the damping is stronger for the spatial  modes with larger spatial frequency.
Indeed, not only the
damping term $\varepsilon^{-1}\, \partial_t \Delta_x u$ in \equ{frequencydamped} is affected by a factor which is
the inverse of the small parameter $\varepsilon$, but it contains also the unbounded operator
$\Delta_x$.

For simplicity, we have considered the case where the $\Delta_x$ appearing in the damping
and in the restoring force are the same operator.
Some slight generalizations are possible, such as taking different
operators for the damping and the restoring force provided they commute. In many physical applications,
it is natural that the operators describing the damping and the force
commute, since they have to be translation invariant and isotropic.

{\bf B) Large stiffness model:} This is a generalization of the model introduced in \cite{Flores13, Flo-Mer-Pel-07}, described
by the equation
\beq{largestiff}
\varepsilon^2 \partial_{tt} u(t,x)+\partial _t u(t,x)-\Delta_x u(t,x) + h(u(t,x),x)=f(\omega t,x)\ ;
\eeq
in \cite{Flores13, Flo-Mer-Pel-07} one can find the specific case
$h(u,x) = \gamma/(1 + u)^2$ with $\varepsilon\in\real$ and with $\gamma\geq 0$ a dimensionless parameter
which provides the relative strengths of electrostatic and mechanical forces.

 Equation \equ{largestiff} models an electrostatically
actuated MEMS (Micro-Electro-Mechanical-Systems) device. Precisely,
the physical interpretation of the model
\eqref{largestiff} is that the restoring force of
the oscillators forming the wave equations is very large.
This type of equations are used to model the deflection of an elastic
membrane suspended above a rigid ground plate, with a voltage source and
a fixed capacitor.
The model represents
the limit of small aspect ratio, when the gap size is small
compared to the device length.
The paper \cite{Flo-Mer-Pel-07} contains a detailed discussion of the motivation.
It is interesting to note that the varactor equation
is somehow a model for the problems considered here.

{\bf B') The modified large stiffness model:}
We modify the non-linearity of model B by assuming that it is of order $\varepsilon$, as described by the
following equation:
$$
\varepsilon^2 \partial_{tt} u(t,x)+\partial _t u(t,x)-\Delta_x u(t,x) + \varepsilon h(u(t,x),x)=f(\omega t,x)\ ;
$$
the above equation appears in the study of MEMS with high aspect-ratios
and/or when the applied tension is high.

\subsection{Regularity assumptions and boundary conditions}
\label{sec:hconditions}
We will require that $f$ is smooth in $x$, satisfies
the boundary conditions, and is analytic in the
variable $\theta\equiv \omega t$.  We will formulate this assumption more
precisely by saying that $f$ belongs to a Hilbert space which we shall call $\A_{\rho,j,m}$;
see the definition in Section~\ref{sec:spaces-findim-Banach}, where
we will impose some restrictions on the  parameters (as we will see $\rho, j$ measure the
analyticity properties and $m$ measures the regularity properties in the space variables).

We will assume that $h$ has some regularity properties too.
Roughly, we will require that
$h$ is analytic in its first argument and differentiable when the second argument ranges over $\D$.
Slightly more precisely, we will require that $h$ is
such that given a function $u \in \A_{\rho,j,m}$,
then $h(u(\theta, x), x)$ is also in $\A_{\rho,j,m}$ and that the
map  $u\mapsto h(u(\cdot),\cdot)$ is differentiable in the sense
of maps in Banach spaces.
We will also require that $h$ satisfies certain geometric conditions ensuring that the
boundary conditions are preserved.
Precisely we make the following requirements.

\bf BCD. \rm For Dirichlet boundary conditions we require that
$h(0,x) = 0$.

\bf BCN. \rm For  Neumann boundary conditions we require that
\begin{equation}
\label{Neumanreq}
n(x)\cdot (D_x h)(u, x ) = 0 \qquad {\rm for\ all\ } x \in \partial \D\ ,\quad u \in \real\ ,
\end{equation}
where $n(x)$ denotes the normal to the domain $\D$ at $x$.
In this way, we obtain that
\[
n(x) \cdot D_x [  h(u(t,x), x) ]=  (D_u h)(u(t,x), x)\ n(x) \cdot D_x u(t,x)
+ n(x)\cdot (D_x h)(u(t,x), x )\ ,
\]
which equals zero if $u(t,\cdot)$ satisfies the Neumann
boundary conditions and \eqref{Neumanreq} holds.  Notice that we have
used that $D_u h$ is one-dimensional.

We anticipate that, besides the above regularity and boundary conditions, we will
also require some non-degeneracy conditions on $h$.

\begin{remark}
In this paper we will construct solutions analytic in time.
The proofs work similarly in
spaces of  functions with Sobolev regularity in time (with high enough Sobolev exponent depending
on the dimension of the frequency), when developing the
theory for finitely differentiable cases (i.e., when the functions $f,h$ are
only assumed to be finitely differentiable).
Of course in this case one can consider only $\varepsilon \in \mathbb{R}$.
\end{remark}

\section{Formulation of the problem and overview of the method for model A}\label{sec:res1}

In this section we go over the method for model $A$ and
reduce it to a fixed point problem. We will present first the formal
manipulations, since they are the motivation for the constructions and
the precise definitions given later on. Notably, the choice of
spaces in Section~\ref{sec:spaces} will be motivated by the need
that the operator appearing in the fixed point equation maps
the spaces into themselves and it is a contraction.

\subsection{Response solutions and formal power series}
Our goal is  to find response solutions of the form
\beq{response}
u_\varepsilon(t,x)=c_0(x)+U_\varepsilon(\omega t,x)\ ,
\eeq
where for each
fixed $\varepsilon$,
 $U_\varepsilon:\torus^d\times\overline{\D}\rightarrow\real$ is at least $O(\varepsilon)$.
We will refer to $c_0$ as the zeroth order term and
omit the index $\varepsilon$ whenever this does not lead to confusion.

We will first show that when we write $U$ as a formal power
series in
$\varepsilon$
\begin{equation}\label{correction}
U_\varepsilon = \sum_{j=1}^\infty \varepsilon^j U_j\ ,
\end{equation}
the coefficients $U_j$ can be formally defined: the appropriate
Banach spaces of functions in which the coefficients actually exist will be specified in
Section~\ref{sec:spaces}. Such Banach spaces will include
regularity properties as well as the boundary conditions.

Inserting \equ{response} in \equ{dissipativewave}, we get that
the function $U_\ep$ must satisfy the equation\footnote{The search of quasi--periodic solutions with frequency $\omega$
having rationally independent components
is equivalent to looking for a solution $u=u(\theta,x)$ of the differential equation in which $\omega t$ is replaced by $\theta$ and  $\partial_t$ is replaced by
$\omega\cdot\nabla_\theta$; this is why we shall study functions of the form $u=u(\theta,x)$.}:
\beq{hullPDE}
(\omega\cdot\nabla_\theta)^2U_\ep(\theta, x)+{1\over\varepsilon}
(\omega\cdot\nabla_\theta)U_\ep(\theta, x)-\Delta_x U_\ep(\theta, x)
-\Delta_xc_0(x)
+h(c_0(x)+U_\ep(\theta, x),x)=f(\theta, x)\ .
\eeq
The solution of \equ{hullPDE} will be the centerpiece of
our treatment. Later, we will develop analogous procedures
for models A', B, B' (see Section~\ref{sec:others}).
We remark that the series expansion \equ{correction} does not contain the term $j=0$; in fact, if we add
a term $U_0$ to the series \equ{correction}, then taking the coefficient of order $\varepsilon^{-1}$ in \equ{hullPDE},
the term $U_0$ would satisfy
$$
(\omega\cdot \nabla_\theta) U_0(\theta,x)=0\ ,
$$
showing that the solution $U_0$, which can be found under the non--resonance assumption on $\omega$,
is independent on $\theta$. However, having written the response function as in \equ{response} with $c_0$ being
the $\theta$-independent part, we conclude that it must be $U_0=0$.

\subsection{Formal solutions of the equation for response functions}
In this section we describe how to
obtain a formal power series solution for
\eqref{hullPDE}. This is  step a) of
the strategy discussed in the introduction.

\subsubsection{Dividing the problem into zeroth order and higher orders}

We introduce the notation:
\begin{eqnarray} \label{linearN}
N_\varepsilon U(\theta, x) &\equiv& [(\omega\cdot\nabla_\theta)^2+{1\over\varepsilon}(\omega\cdot\nabla_\theta)+\Lop] U(\theta, x)\ , \\
\label{Ldefined}
\Lop \eta(x) &\equiv& - \Delta_x \eta(x) +  h'(c_0(x), x) \eta(x)\ ,\\
\label{nonlinearG}
G(U)(\theta,x)&\equiv& h(c_0(x)+U(\theta,x),x)-h(c_0(x),x)- h'(c_0(x),x)\, U(\theta,x)\ .
\end{eqnarray}
Note  that the operator $N_\varepsilon$ depends on $\varepsilon$, whereas
$G$ and $\Lop$ are independent of $\varepsilon$.

\begin{remark}
If the operator $\Lop$ is elliptic and self-adjoint in $L^2_{BC}(\D)$
(namely $L^2(\D)$ with boundary conditions),
then the eigenfunctions constitute a complete set for the
Hilbert space $L^2_{BC}(\D)$ because $\Lop$ has compact resolvent, see
\cite{Helmberg,Davies95}. Similar considerations apply to the analogous
operator $\Lop$ introduced for models A', B.

We will use that the eigenvalues $\lambda_n$ of $\Lop$ are real and that we can characterize
the Sobolev spaces in terms of the coefficients of the eigenfunction
expansions.
\end{remark}

With the above notations and denoting by $\langle \cdot\rangle$ the average with
respect to $\theta$, it is just elementary algebra
to show that the equation \eqref{hullPDE}
is implied by the pair of equations:
\beq{rearranged}
N_\varepsilon  U_\ep(\theta, x) +G(U_\ep) (\theta, x) = f(\theta, x)-\langle f\rangle(x)\ ,
\eeq
\beq{c0}
-\Delta_x c_0(x)+h(c_0(x),x)=\langle f\rangle(x)\ .
\eeq

The reason to divide the equation \equ{hullPDE} into
\equ{rearranged} and \equ{c0} is that \equ{c0} is the leading order
in $\varepsilon$.

Notice that the order $\varepsilon^0$-term in \equ{hullPDE} is
\beq{order0}
(\omega \cdot \nabla_\theta) U_1(\theta,x)-\Delta_xc_0(x)
+h(c_0(x),x)=f(\theta, x)\ .
\eeq
Hence for a solution $U_1$  of \equ{order0} to exist,
it is necessary that the average of \equ{order0} with respect to $\theta$ is
zero (hence equation \equ{c0}). Of course, if $\omega$ satisfies suitable non-resonance conditions and
the functions are smooth, it is indeed possible to obtain $U_1$.
The solution of equation \equ{order0}, which is standard in KAM theory and which can be dealt with
Fourier expansions, will be discussed in Section~\ref{allorders}.
In conclusion the system \equ{rearranged}, \equ{c0} is equivalent to
\eqref{hullPDE}, if we look for formal solutions as in \equ{correction}.

Notice that the system
\equ{rearranged}, \equ{c0} has an upper triangular structure.
In particular, the equation \equ{c0} involves only $c_0$:
once we obtain a solution  $c_0$ of \equ{c0}, we
can substitute it in \eqref{order0} and obtain the solution $U_1$,
and then proceed to higher orders.

The existence of solutions of  \equ{c0} has been
studied extensively in the literature through a great variety of
methods. In Appendix~\ref{sec:c0A} we will present some of the results
available in the literature.

As a result we obtain that  under many circumstances there are
several (often infinitely many)  $c_0$ solving \equ{c0}.
For each of them we will see that (under appropriate non-degeneracy
conditions) we can find a unique solution $U_\varepsilon$ (first
as formal power series and then as analytic function in
a domain).  Hence the upper triangular
system may have many solutions, but the only source of non-uniqueness
is the equation \equ{c0} for $c_0$.

\subsubsection{Preliminary assumptions on the operator $\Lop$}
\label{sec:Lopassumption}

We now specify the spectral properties of
the operator $\Lop$.
In general, to characterize the spectrum of an operator, one needs to specify on
which space it acts. Nevertheless in our case
we assume that the operator is elliptic
and that the domain is compact. In this case the
spectrum is discrete and it is the same in all Sobolev spaces.

The well known reason why the spectrum is independent
on the spaces (\cite{Kato, Helmberg}) is that, when the operator is elliptic, the solutions
 gain regularity: this translates to the fact that the
resolvent is compact on any Sobolev space and hence, for all the Sobolev spaces the spectrum is
just a discrete set of eigenvalues with finite
multiplicity.  Furthermore, again by regularity theory, the
eigenfunctions are very smooth, so that they are
eigenfunctions in all Sobolev spaces. Therefore the spectrum is the same in
all Sobolev spaces. The following assumptions {\bf H1}-{\bf H2} will be
requested for model A as well as for models A', B, B' once the operator $\Lop$ is
suitably defined.

\begin{itemize}
\item[{\bf H1}] The spectrum of $\Lop$  is discrete and its
eigenvalues $\lambda_n$ satisfy:
\begin{equation*}\label{generalspec}
0 \le \lambda_{n} \le  \lambda_{n+1}\ , \quad \forall \ n\geq1
\end{equation*}
and the multiplicity of each eigenvalue is finite, possibly increasing with $n$;
\item[{\bf H2}] The smallest eigenvalue is positive: $\lambda_{1}>0$.
\end{itemize}

In the case of models A', B' we will also assume the following hypotheses
on the operator $-\Delta_x$.

\begin{itemize}
\item[{\bf H1'}] The spectrum of $-\Delta_x$  is discrete and its
eigenvalues $\lambda^\Delta_n$ satisfy
\begin{equation*}\label{generalspec}
0 \le \lambda^\Delta_{n} \le  \lambda^\Delta_{n+1}\ , \quad \forall n\geq1
\end{equation*}
and   the  multiplicity of each eigenvalue is finite, possibly increasing with $n$.
\item[{\bf H2'}] The smallest eigenvalue is positive: $\lambda_1^\Delta>0$.
\end{itemize}

\begin{remark}\label{rmk:H}
Note that a consequence of {\bf H1} and {\bf H1'} is that there is
an orthonormal basis of  eigenfunctions $\Phi_n$ for the operator $P= \Lop \, \, or\, -\Delta_x$ in $L^2(\D)$, such that
\[\label{eigenfuncitons}
P \Phi_n = \lambda_n^{(P)} \Phi_n \qquad \textrm{ for }\ n = 1, 2, ..., \qquad
{\rm with}\ \ \lambda_n^{(P)}\equiv\left\{
\begin{aligned}
&\lambda_n \qquad {\rm if}\ P=\Lop\ ,\\
&\lambda_n^\Delta\qquad {\rm if}\ P=-\Delta_x\ .
\end{aligned}
\right.
\]
\end{remark}

\begin{remark}
We  note that {\bf H1} and {\bf H2} are in turn assumptions on
$h$, $c_0$.
In some arguments, we will need to assume only {\bf H1},
but in order to get the crucial estimates on the ``small divisors''
(and hence to obtain the final result, see Theorem \ref{convergence}),
we will need to assume that there is the spectral gap  in {\bf H2}.
\end{remark}

The above assumptions can be slightly modified; in particular,
the previous assumptions {\bf H1}-{\bf H2}-{\bf H1'}-{\bf H2'} can be
extended to encompass the case of a continuous spectrum (see Remark~\ref{rem:cont} below).

\subsubsection{The nonlinear term and the boundary conditions}

We start by noticing that $G$ is the functional analogue of
the nonlinear term used in \cite{CallejaCL13}, provided of
course that the operator $G$ is defined from some appropriate space to itself.

In this section we will just check that, if we assume
$h$ to satisfy the conditions {\bf BCD}, {\bf BCN} in
Section~\ref{sec:hconditions}, then
the operator $G$ preserves the spaces of functions satisfying these conditions.

In the case of periodic boundary conditions, there is nothing to check.

For  Dirichlet boundary conditions we observe that
if $x \in \partial \D$ and $c_0, U$ satisfy the Dirichlet boundary conditions,
then
$c_0(x) = 0$,  $U(\theta, x) = 0$ and, hence
$
G(U)(\theta, x) = 0$.

For  Neumann boundary conditions we observe that, if
$h$ satisfies \eqref{Neumanreq}, then we just need to check that
\begin{equation}\label{Neumanreqder}
n(x) \cdot D_x [  h'(c_0(x), x)  U(\theta, x) ] = 0 \qquad \forall x \in \partial \D,\ \theta \in \torus^d\ .
\end{equation}
The left hand side of \eqref{Neumanreqder} can be written as the sum of three pieces, i.e.
for $x \in \partial \D$,
\begin{equation}\label{Neumanexp}
\begin{split}
n(x) \cdot &D_x [  h'(c_0(x), x)  U(\theta, x) ]
=   h''( c_0(x), x)   (n(x) \cdot D_x c_0(x))  U(\theta, x)\\
&+n(x) \cdot (D_x h')(c_0(x), x) U(\theta, x)
+h'(c_0(x), x)\ n(x) \cdot (D_xU)(\theta, x)\ .
\end{split}
\end{equation}
The first and second terms in the right hand side of \eqref{Neumanexp}
vanish, since we impose that
$c_0(\cdot)$ satisfies Neumann boundary conditions and
$h$ satisfies \eqref{Neumanreq}.
Therefore, if $U(\theta, \cdot)$ satisfies Neumann boundary conditions,
the last term in \eqref{Neumanexp} will also be equal to zero.

\subsection{The higher order equations}\label{sec:highorder}

We are looking for a formal power series solution of \equ{rearranged}. Assume that we solved \eqref{c0},
insert \eqref{correction} into \eqref{rearranged} and expand the power
series (this requires enough regularity for
the function $h$ which we will make explicit later).

Equating the coefficients of the same power $\varepsilon^N$ for $N \ge 0$, we
obtain the following recursive equations for $N\geq 0$:
\beqa{orderN}
(\omega\cdot\nabla_\theta)U_{N+1}(\theta, x)
&+&
(\omega\cdot\nabla_\theta)^2U_N (\theta, x)-\Delta_x U_N (\theta, x)
+h'(c_0(x),x)\ U_N(\theta, x)\nonumber\\
&=& S_N(c_0(\theta, x), U_1(\theta, x), \ldots, U_{N-1}(\theta, x))\ ,
\eeqa
where $S_N$ is a polynomial expression in $U_1,\ldots, U_{N-1}$ obtained
by applying the Taylor theorem to order $N$ in the equation \equ{hullPDE} and gathering terms.

We  think of \eqref{orderN} as an equation for $U_{N+1}$,
given all the previous terms of the expansion.
Of course  we need to assume that $\omega\cdot k\neq0$ and indeed that it is not too small as $|k|$ increases. Provided that
\begin{equation}\label{compatibility}
\langle\Delta_x U_N (\theta, x)-
h'(c_0(x),x)\ U_N(\theta, x) +  S_N(c_0, U_1, \ldots, U_{N-1})(\theta, x)
\rangle  =
0\ ,
\end{equation}
we can find $U_{N+1}$ which is unique up to the choice of
an additive function of $x$ alone.

Hence, as it is standard when dealing with Lindstedt series, proceeding
by induction we assume that we have determined
$U_1,\ldots U_N$ and then using \eqref{orderN} we can determine $U_{N+1}$ up to
an additive function of $x$: such a function is obtained by solving \eqref{compatibility}
and this can be done because of {\bf H2}.

Sufficient conditions for the existence of an \sl approximate solution \rm provided by a truncation to order $N$ of the
series expansion \equ{correction} are given in Theorem~\ref{allorders}; see also Section~\ref{sec:prescribed}
for a discussion of the existence of an approximate solution to a finite order by solving \equ{orderN}
(compare with \equ{epsilonexpansion} in Section~\ref{sec:prescribed}).

\subsection{Formulation of the fixed point problem
equivalent to \equ{hullPDE}}

As we shall see, the operator $N_\varepsilon$ in \equ{linearN}
is  invertible in the spaces $\A_{\rho, j, m}$ alluded above, if
 $\varepsilon$ ranges in  a suitable domain, so that
\eqref{rearranged} can be rewritten as
\begin{equation} \label{fixed point}
U_\ep(\theta, x) = N_\varepsilon^{-1}\ [-G(U_\ep)(\theta, x) + f(\theta, x)-\langle f\rangle (x) ]
\equiv \T_\ep(U_\ep)(\theta,x)\ ,
\end{equation}
where we have introduced for convenience the operator $\T_\ep$;
we will show that \eqref{fixed point} can be solved by a contraction
mapping argument.

Therefore one of the crucial points of the strategy will be to study the invertibility of
$N_\varepsilon$ and give quantitative estimates on its inverse, notably the Lipschitz constants.
In order to do so,
we provide a uniform lower bound on the eigenvalues of $N_\varepsilon$
(which will depend on $\varepsilon$), using
the assumption {\bf H2} on the eigenvalues of $\Lop$.
By carefully examining such $\varepsilon$-dependent bounds, we
will show that, for $\varepsilon$ in a suitable domain,
the operator  appearing in the right hand side of \eqref{fixed point}
sends a ball centered at the approximate solution (given by the perturbative
expansion) into itself and that it is a contraction inside this ball.  Hence, the fixed
point can be obtained by iteration, starting from the approximate solutions.

We think at the iterative procedure as taking a function analytic in
$\varepsilon$ and producing another analytic function of
$\varepsilon$.
We will show that the convergence is uniform for $\varepsilon$ in a suitably
chosen complex  domain.
Then it is a standard argument that the
limit is an analytic function of $\varepsilon$ in this domain.

The contraction mapping argument is classical; however it requires
to use spaces in which we have sharp estimates, so that we do not
lose any regularity and we obtain that the operator in \equ{fixed point}
sends the spaces into themselves.

Of course, once we have defined
the spaces, we will have to
justify  the formal manipulations, such as  the existence of functional
derivatives. This amounts to making regularity assumptions on the term $h$,
which justify the use of the Taylor's theorem up to order $N$ for
the composition operator.

\subsection{Choice of spaces}
\label{sec:spaces}
In this section we present the
spaces we will use. We discuss some of their elementary properties in Appendix \ref{app.spaces},
where we also add a remark about the continuous spectrum.

The  leading principle is that the norms of the functions can be expressed
in terms of  generalized Fourier coefficients, namely the coefficients associated
to the basis given by the product of the Fourier basis in $\theta$ and the eigenfunctions of $\Lop$
with boundary conditions in $x$.

This principle allows us to estimate rather easily the inverse of the linear
operator $N_\varepsilon$ in \eqref{linearN} just by estimating its eigenvalues, because
we are allowed to use the base in which $N_\varepsilon$ is diagonal.

 We
also need  the spaces to have other properties  allowing us to control
the non-linear terms, such as Banach algebra properties
and properties of the composition operator, so that we can study the operator $G$.
Since we want to obtain analyticity in $\varepsilon$, we will
also need spaces of analytic functions and, in order to simplify the
analysis, we require that they are Hilbert spaces.
Note that we think of the functions in $x$ as  ``scalars'' in analogy to what happens in \cite{CallejaCL13};
hence, it is natural to consider Hilbert spaces of analytic
functions in $\theta$ taking values in another Hilbert space of
functions of $x$.

The choice of the spaces presented here satisfies such properties and leads
to simple proofs. Of course we are not claiming that the choices
we make are optimal and it is quite plausible that other choices
(e.g., analytic functions in both variables) could lead to
better regularity. The main problem in using spaces of analytic functions in $x$
is that it is not clear to us how to express the analyticity of a function in terms of the
coefficients of the expansions in eigenvalues.

We will present several equivalent norms, since some of the properties
of the space will be easier to verify in one norm than in another.
Henceforth, given two (finite or infinite
dimensional) equivalent norms $\|\cdot\|$, $\|\cdot\|'$, we write $\|\cdot\|\cong\|\cdot\|'$ .

\subsubsection{Sobolev spaces with boundary conditions}

In this section we introduce the Sobolev-like spaces which we will
use; we will define them only for
indices $m \in 2 \nat$, since this is enough for
our purposes. The advantage is that, for these indices,
it is possible to give particularly simple characterizations of
the norm in terms of the eigenfunction expansions. Using several characterizations of the norms
allows one to obtain simple proofs of Lipschitz properties of operators.

For functions $S:\overline{\D} \to \complex$
satisfying the corresponding boundary conditions and for $m\in2\nat$,
we define the family of equivalent norms as
\beq{SobolevBC}
\|S\|_{H^m_\Lop} = \|\Lop^{m/2} S\|_{L^2}\ .
\eeq
If $S(x) = \sum_{n=1}^\infty \widehat S_n \Phi_n(x)$ with $\widehat S_n \in \real$ and $\Phi_n$ as
in Remark  \ref{rmk:H}, then the Sobolev norm \eqref{SobolevBC} is given by
\[
\|S\|^2_{H^m_\Lop} = \sum_{n=1}^\infty \lambda_n^m |\widehat S_n|^2\ ,
\]
where $\lambda_n$ are the eigenvalues of $\Lop$ (recall that the $\Phi_n$'s form
a basis of eigenfunctions of $\Lop$).
Since $\Lop$ is elliptic, by G\aa rding's inequality
(see \cite{TaylorII}-Theorem 6.1 of Chap. 7), we have
\beq{equivalenceS}
\|S\|_{H^m_\Lop}\cong \|S\|_{H^m}\ ,
\eeq
where $\|\cdot\|_{H^m}$
is the standard Sobolev norm, namely
\[\|S\|_{H^m} = \|(\Delta_0+1)^{m/2} S\|_{L^2}\]
with $\Delta_0$ the standard constant coefficient Laplacian
and we are considering $S$ satisfying the specified boundary conditions.

The spaces  $H^m_{\Lop}(\D)_{BC}$ and $H^m(\D)_{BC}$
are the completion of $C^\infty_0$ -- the set of $C^\infty$ functions
with compact support contained in the interior of $\D$ --
under the above norms.
For notational convenience we shall not write explicitly the dependence on
the boundary conditions unless needed.

It is well known that for $m>\ell/2$
the Sobolev spaces satisfy the Banach algebra property (\cite{TaylorIII}) and hence the equivalence of the norms
$\|\cdot\|_{H^m_{\Lop}}$ and $\|\cdot\|_{H^m}$ in \equ{equivalenceS} implies for every $S_1,S_2\in H^m_{\Lop}$:
$$
\|S_1 S_2\|_{H^m_\Lop} \leq C  \|S_1\|_{H_\Lop^m} \|S_2\|_{H_{\Lop}^m}\ ,\qquad m>{\ell\over 2}
$$
for some constant $C>0$.

When  $m > \ell/2$ the Sobolev embedding
theorem says that the functions in $H^m$ are continuous, so that the Dirichlet
boundary conditions have classical meaning.

Similarly, when $m > \ell/2+1$, the gradient of functions in
$H^m$ are continuously differentiable. Hence, the Neumann boundary
conditions have classical meaning.

\subsubsection{Spaces of analytic functions of complex variables
taking values into Banach   spaces}
\label{sec:spaces-findim-Banach}

We introduce
domains that consist of a strip around the torus $\torus^d$
in the imaginary
direction. We will consider analytic functions in these domains.

\begin{definition}\label{complexstrip}
Given $\rho > 0$, we denote by $\torus_\rho^d$ the set
\[\torus_\rho^d = \left \{\theta \in (\complex/\integer)^d\,:\, {\rm Re}(\theta_j)
\in \torus\ ,\ \  |\Im(\theta_j)|\leq \rho\ ,\ \  j = 1, ... ,d\right \}\ .\]
\end{definition}

When we consider functions of $\theta \in \torus_\rho^d$ and $x\in\overline{\D}$, we
can think of them as functions from $\torus^d_\rho$ into $H^m_\Lop$ which are analytic
\footnote{
When we consider domains which are closed with  a smooth boundary, we
refer to analytic functions as functions which are
analytic in the interior and that extend continuously to the boundary. For us, using domains which are
compact is slightly more convenient in order to quote embedding theorems, etc.
Nevertheless, in order to avoid repetitions, we omit that analyticity is meant only for
the interior and that we assume the extension to the boundary.}. The spaces which we will consider
are the standard Bargmann spaces taking values into $H^m_\Lop$.

Given a function
$u=u(\theta,x)$  which we expand as
\beq{generalizedFourier}
u(\theta, x) = \sum_{k\in\integer^d} e^{2 \pi i k\cdot\theta} \hat u_k(x)
= \sum_{k\in\integer^d,n\geq 1} e^{2 \pi i k\cdot\theta} \Phi_n(x) \hat u_{k,n}\ ,
\eeq
we will consider the space of analytic functions of
$\theta$ endowed with the $H^j(\torus^{d}_\rho; H^m_\Lop)$ norm defined below.
We emphasize that we are considering $\torus^d_\rho$ as a $2d$-dimensional real
manifold with boundary. Again, for simplicity, we just consider the even Sobolev exponents $j$.

Precisely, for $\rho>0$, $j,m \in 2 \nat$,
setting $\Delta_\theta\equiv \sum_{n=1}^d \nabla_{\theta_n}\nabla_{\bar\theta_n}$
(where the bar denotes complex conjugation), we define the $H^j(\torus^d_\rho;H^m_\Lop)$ norm as:
\beq{analyticnorm}
\begin{split}
\|u\|^2_{\rho,j, m} & = \int_{\torus^d_\rho} \|(\Delta_\theta +1)^{j\over 2} u(\theta, \cdot)  \|_{H^m_\Lop}^2\ d^{2d} \theta  \\
&= \int_{\torus^d_\rho} \|\Big( \sum_{n=1}^d \nabla_{\theta_n}  \nabla_{\bar \theta_n} +1\Big)^{j\over 2}
u(\theta, \cdot) \|_{H^m_\Lop}^2\ d^{2d} \theta\ .  \\
\end{split}
\eeq
We denote by $\A_{\rho,j, m}$ the space of functions analytic in $\theta$ whose norm $\|\cdot\|_{\rho,j,m}$ is finite.
Note that $\A_{\rho,j,m}$ are Hilbert spaces, since the norm \eqref{analyticnorm}
clearly comes from the
inner product
$$
\langle u, v \rangle = \int_{\torus_\rho^d} \langle u,
(\Delta_\theta + 1)^j v \rangle_{H^m}\ d^{2d} \theta\ .
$$
Moreover they are complete, since the limit in the $\|\cdot\|_{\rho,j,m}$-norm
of analytic functions is an analytic function (\cite{Reed-Simon-Vol1}).

\begin{remark}
We think of a function $u \in \A_{\rho,j, m}$ as an analytic function
from
$\torus^d_\rho$ into $H_\Lop^m$, say $\theta \to u(\theta, \cdot)$.
In this way, the problems considered here look closer to the formulation of the varactor
problem considered in \cite{CallejaCL13}. The PDE looks
formally like an ODE in $H^m_\Lop$ and the
response solutions will be analytic functions from the torus
into $H^m_\Lop$. For sufficiently high $m$, these will be classical functions which are
analytic in the $t$ variable and differentiable in the variable $x$. Hence, they will be
classical solutions for the PDE.
\end{remark}

\section{Precise statement of the results}\label{sec:main}

\subsection{Approximate solutions of the fixed point problem}

For all models we can give a definition of ``approximate solution" as follows.

\begin{definition}\label{def:approx}
Let us consider a family of functional equations
\begin{equation}\label{supergeneral}
{\mathcal F}_\ep(U)=0\ ,
\end{equation}
where ${\mathcal F}_\ep:{\mathcal A}_{\rho,j,m}\rightarrow{\mathcal
A}_{\rho,j,m}$ is an operator that maps $\ep$-dependent families into
families (of course the operator ${\mathcal F}_\ep$ may have an explicit $\ep$-dependence).
We say that
\begin{equation}\label{truncated}
U_{\varepsilon}^{(M)}=\sum_{k=0}^M\varepsilon^k U_k
\end{equation}
for some $M\in\integer_+$ is an approximate solution up to order $M$
of \eqref{supergeneral}, if
$$
\| {\mathcal F_\ep}(U_\varepsilon^{(M)})\|_{\rho,j,m}=O(\varepsilon^{M+1})\ .
$$
\end{definition}

\subsection{Main results}

Our main results are provided by the following
Theorems~\ref{thm:main}, \ref{thm:allorders}, \ref{convergence}.

Theorem~\ref{thm:main} is based on a contraction mapping argument and it states the existence of a
solution, provided we assume the existence of an approximate solution.

Theorem~\ref{thm:allorders} gives  sufficient
conditions for the existence of an approximate solution up
to any order under some non--resonance assumptions on the frequency: the higher is the
order of approximation we want to reach, the more restrictive will be the condition on the frequency.

Theorem~\ref{convergence}  summarizes the results above, i.e. it gives the existence of an analytic
 solution under the requirements of
Theorem~\ref{thm:allorders}, which provides the approximate solution.
The proof of Theorem~\ref{convergence} relies on applying Theorem~\ref{thm:main}
to the approximate solutions provided by Theorem~\ref{thm:allorders}.

In the Theorems \ref{thm:main}, \ref{thm:allorders}, \ref{convergence} we will
assume that $f$ belongs to the space of functions $\A_{\rho,j,m}$ as in Proposition~\ref{properties}, which
ensures the validity of the Banach algebra property.

\begin{theorem}\label{thm:main}
Assume that $f$ is in $\A_{\rho,j,m}$ for $\rho>0$, $j,m \in 2 \nat$, $j>d$, $m>\ell/2$
($m>\ell/2+1$ in the case of Neumann boundary conditions).
Let $h:\mathcal{B} \times \overline{\D} \to \complex$ with $\mathcal{B}\subset\complex$ open set,
and let $\D$ be either of the form {\bf D1} or {\bf D2} as in Section~\ref{sec:PDEs}. We assume that $h$
is analytic in $\mathcal{B}$ and $C^m(\D)\cap C(\overline{\D})$ in $x$.

Consider the models A, A', B, B' with D, N, P boundary conditions; assume that the hypotheses
{\bf H1}-{\bf H2} are satisfied (see Section~\ref{sec:Lopassumption}),
and that the non-linearity $h$ satisfies BCD or BCN (See Section~\ref{sec:hconditions}) in case of D or N boundary conditions, respectively (depending on the boundary conditions considered for the equation).
For models A', B' assume also {\bf H1'}-{\bf H2'}.

For model A
assume that the zeroth order term $c_0$ (see \equ{c0})
admits a solution (some sufficient conditions are given in Appendix~\ref{sec:c0A}) and that
for some $M\in\mathbb{N}$, $M\geq 2$, there exists an approximate solution in $\varepsilon$ of \equ{hullPDE} up
to order $M$.

Let $\varepsilon$ be in the domain $\Omega_B = \cup_\sigma \Omega_{\sigma,B}$ with $B>B_0$
for some $B_0>0$ sufficiently large, $\sigma>0$ sufficiently small, where
\begin{equation}\label{parabolicdomain}
\Omega_{\sigma,B} \equiv \{\varepsilon=\xi+i\eta \in \complex\, : \, \xi>
B\, \eta^2 \, , \, \sigma < |\varepsilon| < 2 \sigma\}
\end{equation}
and $\theta$ in the strip of size $\rho>0$
$$
\torus_\rho^d \equiv \{ \theta \in (\complex/\integer)^d \, : \,
{\rm Re}(\theta_j)\in\torus\ ,\quad |\Im(\theta_j)|\leq\rho\ ,
\quad j=1,...,d\}\ .
$$
Then, there exists a function $U_\ep=U_\ep(\theta,x)\in\A_{\rho,j,m}$, which provides
an exact solution of \equ{hullPDE}.

For model A'
assume that the zeroth order term $c_0$ (see \equ{c0})
admits a solution (some sufficient conditions are given in Appendix~\ref{sec:c0A}) and that
for some $M\in\mathbb{N}$, $M\geq 2$, there exists an approximate solution in $\varepsilon$ of
\beq{hullPDEC}
(\omega\cdot\nabla_\theta)^2 U_\ep(\theta,x)+{1\over\varepsilon} (\omega\cdot\nabla_\theta)\ \Delta_x U_\ep(\theta,x)-\Delta_x U_\ep(\theta,x)
-\Delta_x c_0(x)+h(c_0(x)+U_\ep(\theta,x),x)=f(\theta,x)\ .
\eeq
Let $\varepsilon$ be in a domain of the form \equ{parabolicdomain}. Then,
there exists a function $u=u(\theta,x)=c_0(x)+U_\ep(\theta,x)$ as in \equ{response}, belonging to
$\A_{\rho,j,m}$, which provides an exact solution of \equ{hullPDEC}.

For model B
we assume
that there exists an approximate solution of
\beq{hullPDEB}
\varepsilon^2 (\omega\cdot\nabla_\theta)^2 U_\ep(\theta,x)+(\omega\cdot\nabla_\theta)U_\ep(\theta,x)-\Delta_x U_\ep(\theta,x)
+h(U_\ep(\theta,x),x)=f(\theta,x)
\eeq
up to order $M$ with $M\in\mathbb{N}$, $M\geq 2$.
Assuming that $\varepsilon$ belongs to the domain
\beq{omd}
\Omega_\delta\equiv \{\varepsilon=\xi+i\eta \in \complex\, : \, {\rm Re}(-\varepsilon^2)\geq\delta\}\cup\{
\varepsilon = \xi \in \real \,:\, \delta<|\xi|<2\delta\}
\eeq
for some $\delta>0$, then
there exists a function $U_\ep=U_\ep(\theta,x)\in\A_{\rho,j,m}$,
which provides an exact solution of \equ{hullPDEB}.

For model B'
assume that the zeroth order term
admits a solution (see Section~\ref{sec:c0Bp}) and that
for $M\in\mathbb{N}$, $M\geq 2$, there exists an approximate solution up to order $M$
in $\A_{\rho,j,m}$ of
\beq{hullPDEBp}
\varepsilon^2 (\omega\cdot\nabla_\theta)^2 U_\ep(\theta,x)+(\omega\cdot\nabla_\theta)U_\ep(\theta,x)-\Delta_x U_\ep(\theta,x)
+\varepsilon h(U_\ep(\theta,x),x)=f(\theta,x)\ .
\eeq
Assuming that $\varepsilon$ belongs to $\Omega_\delta$ as in \equ{omd} for some $\delta>0$, then there
exists a function $U_\ep=U_\ep(\theta,x)\in\A_{\rho,j,m}$, which provides an exact solution of \equ{hullPDEBp}.

In all the cases above, the solution $U_\ep$ is analytic in the considered domains as a function of $\varepsilon$ and it is
asymptotic to the approximate solution.

\end{theorem}

Note that Theorem~\ref{thm:main} involves mainly regularity assumptions
and the requirement that there exist approximate solutions at least of order $2$.
Sufficient conditions for the existence of an approximate solution
(given by the expansion to any arbitrary order) are provided by the following result.

\begin{theorem}\label{thm:allorders}
Assume that $f$ is in $\A_{\rho,j,m}$ for $\rho>0$, $j,m \in 2 \nat$, $j>d$, $m>\ell/2$
($m>\ell/2+1$ in the case of Neumann boundary conditions).
Let $h:\mathcal{B} \times \overline{\D} \to \complex$ with $\mathcal{B}\subset\complex$ open set,
and let $\D$ be either of the form {\bf D1} or {\bf D2}. We assume that $h$
is analytic in $\mathcal{B}$ and $C^m(\D)\cap C(\overline{\D})$ in $x$.

Consider the models A, A', B' with either D, N, P boundary conditions and assume that $h$
satisfies either BCD, BCN in case of D, N boundary conditions, respectively (depending on the boundary
condition considered for the equation). Assume that the zeroth order term admits a solution
(see, respectively, Appendix~\ref{sec:c0A} for models A, A' and Section~\ref{sec:c0Bp} for model B'). Furthermore:

$i)$ Assume that there exists $M \in \nat$, such
that
\begin{equation}\label{nonresonance}
|k|^{-1}  \log |\omega \cdot k|^{-1} \leq \frac{2 \pi \rho}{M}
\quad \quad \forall k \in \mathbb{Z}^d \setminus \{0\}\ .
\end{equation}
Then, there exists an approximate solution in $\varepsilon$ of \equ{hullPDE}, \equ{hullPDEC}, \equ{hullPDEBp} up to order $M$.
In particular, if
\beq{nonresonance2}
\limsup_{|k|\to \infty} |k|^{-1} \log(|\omega \cdot k|^{-1}) = 0\ ,
\eeq
then we can obtain a formal power series in $\varepsilon$ (whose coefficients are
well defined) solving the equation up to all orders.

$ii)$ If we assume that $f$ is a trigonometric polynomial,
then there is a formal power series in $\varepsilon$  which is a solution
of \eqref{hullPDE}, \equ{hullPDEC}, \equ{hullPDEBp} up to all orders,
without requiring any non-resonance bound on the frequency $\omega$.
\end{theorem}

\begin{remark}
Recall that we are assuming that $\omega \cdot k = 0$, $k\in\integer^d$, implies $k = 0$.
\end{remark}

\begin{remark}
\label{nonresonanceoptimal}
It seems likely that the condition \eqref{nonresonance}
(which is even weaker than the Bryuno condition) is
optimal. Each step of the computation of the perturbative expansion
involves solving a differential equation with $(\omega \cdot k)^{-1}$ as small
divisors. Hence, we expect that the solutions, in general, lose a domain of
definition of size $\rho/M$ at each step.
\end{remark}

Combining Theorems~\ref{thm:main} and \ref{thm:allorders} we obtain the following result.

\begin{theorem}\label{convergence}
Assume that $f$, $h$ satisfy the assumptions of Theorem~\ref{thm:main}
regarding the regularity, the boundary conditions and the existence of the zeroth order solution.
Fix $M\geq 2$ and assume \equ{nonresonance}. Then, we have the following results.

For model A there exists a solution of equation \equ{hullPDE}, which is analytic in $\varepsilon$ and $\theta$, and satisfies the D, N or P boundary conditions.
The analytic solution exists for $\varepsilon$ in the
domain $\Omega_B$ as in Theorem~\ref{thm:main} and $\theta$ in $\torus_\rho^d$.

For model A' assuming {\bf H1'}, {\bf H2'}, there exists an analytic solution of \equ{hullPDEC} in
$\theta\in\torus_\rho^d$ and
$\varepsilon$ with $\varepsilon$ in a domain of the form $\Omega_B$ as in Theorem~\ref{thm:main}.

For model B with D, N, P boundary conditions, provided the existence of an approximate solution,
there exists an analytic solution of \equ{hullPDEB} in $\theta\in\torus_\rho^d$ and
$\varepsilon$ in a domain of the form $\Omega_\delta$ as in \equ{omd} for some $\delta>0$.

For model B' with D, N, P boundary conditions, assuming {\bf H1'}-{\bf H2'}
there exists an analytic solution of \equ{hullPDEBp} in $\theta\in\torus_\rho^d$ and
$\varepsilon$ in a domain of the form $\Omega_\delta$ as in \equ{omd} for some $\delta>0$.
\end{theorem}

\begin{remark}
Note that the non-resonance condition which we need to impose on $\omega$
to obtain the existence to all orders is more restrictive  than the
non-resonance condition we need to obtain the existence of
an analytic solution defined in the domain $\Omega_{\sigma,B}$ or $\Omega_\delta$.
Since, as we argued in Remark~\ref{nonresonanceoptimal}, we believe that
the conditions are optimal, it seems that given an $M_0$
and an $\omega$ that satisfy \eqref{nonresonance}
for $M = M_0$, but not for $M = M_0 +1$, then we can obtain functions
that have an expansion up to order $M_0$, but not $M_0 +1$.

It seems, therefore, possible to arrange the existence of
models with a solution analytic in a domain of the form $\Omega_{\sigma,B}$ or $\Omega_\delta$, but
without Taylor expansion beyond a certain order.
\end{remark}

\begin{remark}
By restricting the domain in Theorem \ref{convergence}, we can obtain stronger contraction properties
for the operator $\T$ defined in \equ{fixed point}.

For example, for model A one possibility would be to consider only one of the domains
$\Omega_{\sigma,B}$ defined in \eqref{parabolicdomain} with $\sigma$ small enough.
Another possibility is to consider conic domains defined as
$$
\Upsilon_{\delta, \sigma} = \{ \ep \in \complex \, : \, |\Im(\ep)|/|\ep| <  \delta\ ,\quad
\sigma < |\ep| < 2\sigma\}
$$
for some $\delta$, $\sigma>0$.
We refer to \cite{CallejaCL13} for further details on the study of the solution of a forced strongly dissipative
ODE on conic domains. By restricting to the real line, it seems possible to obtain results for finitely differentiable
non-linearities.
\end{remark}

\section{Existence of solutions of the zeroth order term}\label{sec:c0}

The existence of solutions of the zeroth order equation for models A, A'
has a very extensive literature and can be done by a variety
of methods; with reference to classical textbooks like \cite{struwe,amb-malc,precup},
we defer the presentation of some of such results in Appendix \ref{sec:c0A}.
For models A, A', we will produce several (even infinitely many)
solutions of the zeroth order equation.

In this section we confine ourselves to discussing models B, B', using arguments based on an implicit
function theorem in Banach spaces. However, we mention that there are
other possibilities which we have not covered, for example methods based on index
theory (\cite{FonsecaG95,Berger73}). We will show that each of these solutions of the zeroth
order equation continues into a formal solution to all orders and that, furthermore, it can be
modified to be a true solution.

\subsection{Solution of the zeroth order term for model B}\label{sec:c0B}

The zeroth order equation for model B is:
\begin{equation}\label{c0U0}
(\omega\cdot\nabla_\theta)U_0(\theta,x)-\Delta_x U_0(\theta,x)+h(U_0(\theta,x),x)=f(\theta,x)\ .
\end{equation}
We will solve equation \equ{c0U0} by reducing it to a fixed point problem
and providing conditions that ensure solvability.

Let us denote by $\Gamma$ the operator
\beq{Gamma}
\Gamma\equiv \omega\cdot\nabla_\theta+\Lop\ ,
\eeq
where $\Lop$ is given by
\beq{LmodelB}
\Lop \eta \equiv -\Delta_x\eta  +h'(0,x)\eta\ .
\eeq
We assume that
$$
h(0, x)=0\ .
$$
Then, \equ{c0U0} becomes
\beq{c0U02}
\Gamma U_0(\theta,x)+G(U_0(\theta,x),x)=f(\theta,x)
\eeq
with $G$ given by
\begin{equation}\label{GmodelB}
G(U_0(\theta,x),x) = h(U_0(\theta,x),x) - h'(0,x)U_0(\theta,x)\ .
\end{equation}
We notice that $\Gamma$ is a diagonal operator in the Fourier basis,
since it is separated in the sum of two parts, one of which acts only
on $\theta$ and the other acting only on $x$.
If we assume that $\Lop$ satisfies {\bf H1} and {\bf H2} and we
denote its eigenvalues by $\lambda_n$, then
$$
\Gamma(e^{2\pi ik\cdot\theta}\Phi_n)=(2\pi i \omega\cdot k+\lambda_n)e^{2\pi ik\cdot\theta}\Phi_n\ .
$$
Thus, we notice that $\Gamma$
is invertible and we can reduce \eqref{c0U02} to the following fixed point problem:
\beq{c0U03}
U_0(\theta,x)= -\Gamma ^{-1}\ G(U_0(\theta,x),x)+\Gamma ^{-1}\ f(\theta,x)\ .
\eeq
We define the operator $\Tau$ by
\beq{TaudefmodelB}
\Tau(U)\equiv -\Gamma ^{-1}\ G(U(\theta,x),x)+\Gamma ^{-1}\ f(\theta,x)\ .
\eeq
Using properties of composition of functions, see Proposition~\ref{prop:composition} in Appendix~\ref{app.spaces}, we can show that
for $U_0, V_0 \in \A_{\rho, j, m}$, then
$\Tau$ satisfies the inequality:
\beqano
\|\Tau(U_0)-\Tau(V_0)\|_{\rho,j,m}&=&\|\Gamma ^{-1}\ h(U_0,x)-\Gamma ^{-1}\ h(V_0,x)-\Gamma ^{-1}h'(0,x)U_0+\Gamma ^{-1}h'(0,x)V_0\|_{\rho,j,m}\nonumber\\
&\leq& C\alpha_0\ \|\Gamma ^{-1}\|_{\rho,j,m}\ \|U_0-V_0\|_{\rho,j,m}\ ,
\eeqano
where the Lipschitz constant of the composition with $G$
is bounded by a constant $C$ times $\alpha_0$. Finally, if we choose $\alpha_0$ small enough so that
$\Tau$ is a contraction, we obtain a solution $U_0$ in a ball of radius $\alpha_0$.
Thus, we have proven the following result.

\begin{proposition}\label{pro:zerob}
Assume that $f$ is in $\A_{\rho,j,m}$ for $\rho>0$, $j,m \in 2 \nat$, $j>d$, $m>\ell/2$
($m>\ell/2+1$ for Neumann boundary conditions).
Let $h:\mathcal{B} \times \overline{\D} \to \complex$ with $\mathcal{B}\subset\complex$ open set,
and let $\D$ be either of the form {\bf D1} or {\bf D2}. We assume that $h$
is analytic in $\mathcal{B}$ and $C^m(\D)\cap C(\overline{\D})$ in $x$ and let $h(0,x)=0$.

Consider model B with either D, N, P boundary conditions and assume that $h$
satisfies, respectively, BCD, BCN in case of D, N boundary conditions (depending on the boundary
condition considered for the equation).
Then, the zeroth order term of model B given by equation \eqref{c0U0}
admits a solution $U_0$ contained in a ball around the
origin in $\A_{\rho,j,m}$ of small enough radius $\alpha_0$ .
\end{proposition}

\subsection{Solution of the zeroth order term for model B'}\label{sec:c0Bp}
The zeroth order equation for model B' is
\begin{equation}\label{U0Bp}
(\omega\cdot\nabla_\theta)U_0(\theta,x)-\Delta_x U_0(\theta,x)=f(\theta,x)\ ,
\end{equation}
which can  be solved under the assumptions {\bf H1'}-{\bf H2'}.
In fact, defining the operator $\Gamma$ acting on $U$ as in \equ{Gamma}, but with $\Delta_x$ instead of $\Lop$,
we can write \equ{U0Bp} as
$$
\Gamma U_0(\theta,x)=f(\theta,x)\ ,
$$
which can be solved, because $\Gamma$ is invertible (using {\bf H1'}-{\bf H2'}).

Indeed, the operator $\Gamma$ is diagonal in the Fourier basis. Let us expand $U_0$ as
$$
U_0(\theta,x)=\sum_{k\in\integer^d}\sum_{n\ge0}e^{2\pi ik\cdot\theta}\Phi_n(x)\tilde U_{0,k,n}
$$
for some coefficients $\tilde{U}_{0,k,n}$; in a similar way, let
$$
f(\theta,x)=\sum_{k\in\integer^d\backslash\{0\}}\sum_{n\ge0}e^{2\pi ik\cdot\theta}\Phi_n(x) f_{k,n}\ .
$$
Then, we obtain:
\beq{Uf}
\tilde{U}_{0,k,n}=\frac{f_{k,n}}{(2\pi i\omega\cdot k+\lambda^\Delta_n)}
\eeq
for $k\ne0$ (with $\lambda_n^\Delta\in\real$ denoting the eigenvalues of $-\Delta_x$) and $\tilde U_{0,0,n}=0$.
From \equ{Uf} we see that the assumptions {\bf H1'}-{\bf H2'} and the regularity of $f(\theta,x)$ imply
the regularity of $\tilde U_0$. Because of {\bf H1'}-{\bf H2'}, we have that $|2\pi i \omega\cdot k+\lambda_n^\Delta|\geq\nu$
for some $\nu>0$; using the character of the norms in Fourier coefficients, we obtain the desired result.

\section{Proof of Theorems \ref{thm:main},
\ref{thm:allorders} and \ref{convergence} for model A}\label{sec:proof}

In this section we present detailed arguments that complete the
proof of Theorems \ref{thm:main}, \ref{thm:allorders} and \ref{convergence} for the case of
the dissipative wave equation \equ{dissipativewave} of model A; in
Section~\ref{sec:others} we will present the necessary modifications in order to have
the results for models A', B, B'.

We start by proving the existence of an approximate
solution up to prescribed orders (Section~\ref{sec:prescribed}); then, we bound the
operator $N_\varepsilon$ in \equ{linearN} providing estimates in a parabolic domain
(Section~\ref{sec:bounds}) and we conclude by showing the existence of a solution of \equ{rearranged}
through a fixed point argument (Section~\ref{sec:fixed}).

\subsection{Existence of an approximate solution up to prescribed orders}\label{sec:prescribed}
\label{allorders}
In this section we describe the construction of the approximate
solution up to a prescribed order $M$, as in the statement of
Theorem~\ref{thm:allorders}.

For model A described by equation \equ{dissipativewave}, the first order term $c_0$ of the expansion
of the response solution (see \equ{response}) satisfies the  semilinear second order elliptic equation \equ{c0}.

To perform the formal manipulations that lead to the approximate solution
up to order $M$, we find it convenient to write the equation \equ{hullPDE} as
\beq{main-expansion}
[\varepsilon(\omega\cdot\nabla_\theta)^2+(\omega\cdot\nabla_\theta)-\varepsilon\Delta_x ]
U_\ep(\theta, x)
-\varepsilon\Delta_x c_0(x)+\varepsilon h(c_0(x)+U_\ep(\theta, x),x)=\varepsilon f(\theta,x)\ .
\eeq
We assume that a solution $c_0$ for \equ{c0} can be found as described in Appendix~\ref{sec:c0A}.
Next, we write formally $U_\ep\equiv U_\varepsilon(\theta)$ in powers of $\varepsilon$ (see \eqref{correction}).
We now show that we can define an approximate solution
of \eqref{main-expansion} as a finite truncation of \eqref{correction}
up to order $M$. Inserting \eqref{correction} into \eqref{main-expansion}, we get
\beq{epsilonexpansion}
\sum_{j=1}^\infty \varepsilon^j [\varepsilon (\omega\cdot\nabla_\theta)^2
+ (\omega \cdot \nabla_\theta) - \varepsilon \Delta_x ] U_j(\theta, x)
-\varepsilon \Delta_x c_0(x) + \varepsilon h(c_0(x) + U_\varepsilon(\theta,x),x)
-\varepsilon f(\theta,x)=0\ .
\eeq
Hence the first order in $\varepsilon$ in \equ{epsilonexpansion} is given by \eqref{order0}.
Since $c_0$ satisfies \equ{c0},
then the equation for the first order in $\varepsilon$ becomes
\beq{firstordereps}
(\omega \cdot \nabla_\theta)U_1(\theta, x) = f(\theta,x)
-\langle f\rangle (x)
\eeq
and it is easy to see that by the non-resonance condition
\equ{nonresonance},
equation \eqref{firstordereps}
has a solution in $\mathcal{A}_{\rho',j,m}$ for some $\rho'<\rho$: a proof of this fact can be found in \cite{CallejaCL13}
for the case of the varactor equation (see also \cite{CFG2}) and can be straightforwardly extended to the
present situation. In fact, let us define $g(\theta,x)\equiv f(\theta,x)-\langle f\rangle (x)$; let us expand
$U_1$ and $g$ as
$$
U_1(\theta,x)=\sum_{k\in\integer^d} \sum_{n\geq 0} e^{2\pi ik\cdot\theta} \Phi_n(x)\ \tilde U_{1,k,n}\ ,\qquad
g(\theta,x)=\sum_{k\in\integer^d\backslash\{0\}} \sum_{n\geq 0} e^{2\pi ik\cdot\theta} \Phi_n(x)\ \tilde g_{k,n}
$$
for suitable coefficients $\tilde U_{1,k,n}$, $\tilde g_{k,n}$. From \equ{firstordereps} we obtain that
$$
\tilde U_{1,k,n}={{\tilde g_{k,n}}\over {2\pi i\omega\cdot k}}\ ,
$$
which is well defined thanks to \equ{nonresonance}. The appearance of the small divisors is the origin of the
loss of analyticity domain.

Note that $U_1(\theta,x)$, as a solution of
\eqref{firstordereps}, has a free parameter, namely
 $\langle U_1\rangle (x)$, which is determined
at the subsequent order.

Recalling the definition \equ{Ldefined} of $\Lop$, the order $\varepsilon^2$ in \equ{epsilonexpansion} is
$$(\omega\cdot \nabla_\theta) U_2(\theta,x) =
- [(\omega \cdot \nabla_\theta)^2 + \Lop] U_1(\theta,x)\ , $$
which admits a solution $U_2(\theta,x)$, provided that the average of the right hand side is zero.
In particular, the average of $U_1$ must satisfy the equation $\Lop(\langle U_1\rangle)=0$ and since
$\Lop$ satisfies {\bf H1}-{\bf H2}, then $\langle U_1\rangle = 0$.

Now, by using the non-resonance condition \equ{nonresonance} up to order
$M$, one can proceed recursively to compute the functions $U_j$ in $\mathcal{A}_{\rho,j,m}$
up to order $M$.
We have that for any $N\leq M-1$ the function $U_{N+1}$ satisfies a
recursive relation of the form \eqref{orderN}.
Again, we must require that the average of the right hand side of \equ{orderN} is zero, which provides
the average of $U_{N}$, i.e.
\beq{un}
\Lop (\langle U_{N} \rangle)  = - \langle S_N(c_0,U_1...,U_{N-1}) \rangle\ .
\eeq
The existence of $\langle U_N \rangle$ satisfying \equ{un} is guaranteed by assumptions {\bf H1}-{\bf H2}, since
the spectrum of $\Lop$ is bounded away from zero.

In this way we obtain the approximate solution up to order $M$
under the assumption \equ{nonresonance}. It follows that we obtain a
well defined approximate solution
to all orders under the condition \equ{nonresonance2}; for instance one can adapt the argument given
in Appendix H of \cite{CG}. Finally, if we assume further that
$f$ is a trigonometric polynomial of degree $J>0$, then $U_N$ is
a trigonometric polynomial of degree $NJ$ and we can obtain the formal solution up to any order $N$.
This concludes the proof of Theorem~\ref{thm:allorders}.

\subsection{Bounds on the operator $N_\varepsilon$}\label{sec:bounds}
In this section we will obtain bounds on $N_\varepsilon$ in appropriate spaces, when $\varepsilon$
is contained in the domain $\Omega_B$ defined
as the union over $\sigma$ of the domains \eqref{parabolicdomain}.
We will also remark that it is not possible to obtain the same bounds when
$\varepsilon$ is on the imaginary axis: indeed, we present separate arguments
that lead us to conjecture that the bound on the spectrum of $N_\varepsilon$ cannot be obtained
when $\varepsilon$ is imaginary.

Note that, since $\Lop$ acts on the $x$ variable
only and $\omega\cdot\nabla_\theta$ on the $\theta$ variable, we can apply separation of variables
and obtain that the spectrum of $N_\varepsilon$ is
$$
\lambda_{n,k}\equiv\lambda_{n,k}(\varepsilon)=-(2\pi\omega\cdot k)^2+{{2\pi i}\over\varepsilon}
(\omega\cdot k)+\lambda_n\ .
$$
As already pointed out, the invertibility of $N_\varepsilon$ will follow from the fact
that its spectrum is bounded
away from zero. For a general operator, the bounds on the inverse would need not only
to estimate the spectrum, but also the spectral projections,
though this is trivial in this case since $\Lop$ is self-adjoint and
$\omega\cdot\nabla_\theta$ is anti self-adjoint, so that $\Lop + \omega\cdot\nabla_\theta$
is a normal operator.

We study the spectrum of $N_\varepsilon$, when $\varepsilon$ ranges in the domain
\begin{equation}\label{omegalpha}
\Omega_{\sigma,B,\alpha} \equiv \{\varepsilon=\xi+i\eta \in \complex\, : \, \xi>
B\, \eta^\alpha \, , \, \sigma < |\varepsilon| < 2 \sigma\}\ .
\end{equation}
Afterwards, we will fix $\alpha$ in such a way that we can use the fixed point argument of Section~\ref{sec:fixed}
and it will turn out that the best choice is $\alpha=2$, thus leading to defining the solution in the domain
$\Omega_{\sigma,B,2}=\Omega_{\sigma,B}$ as defined in \equ{parabolicdomain}.
To study the spectrum of $N_\varepsilon$, we will use the maximum principle for $\lambda_{n,k}$ as a function
of $\varepsilon$; hence, we will get a lower bound on $|\lambda_{n,k}|$ on the boundary of the domain.

Let us suppose that $\lambda_n \to \infty$ whenever $n\to \infty$.
We will show later that this assumption can be relaxed to encompass the case that the sequence of
the $\lambda_n$'s has a finite supremum, even if
this case does not appear in the applications we have in mind. Thus,
we can fix $K\in\nat$ large enough so that $\lambda_K-1\ge\lambda_1$; we will first provide a bound
for $1\leq n \leq K$ on the whole region $\Omega_{\sigma,B,\alpha}$.

For every $n\in \nat$, we want to
estimate $\inf_{k\in \integer^d\setminus\{0\}}|\lambda_{n,k}(\varepsilon)|^2$ for $\varepsilon\in\Omega_{\sigma,B,\alpha}$;
therefore, we study the behavior of
\beqano
|\varepsilon \lambda_{n,k}(\varepsilon)|^2&=&|-\varepsilon (2\pi\omega\cdot k)^2+i(2\pi\omega\cdot k)
+\varepsilon \lambda_n|^2\nonumber\\
&=&|-(\xi+i\eta)(2\pi\omega\cdot k)^2+i(2\pi\omega\cdot k)+(\xi+i\eta)\lambda_n|^2\nonumber\\
&=&\xi^2\ [-(2\pi\omega\cdot k)^2+\lambda_n]^2+[-\eta(2\pi\omega\cdot k)^2+(2\pi\omega\cdot k)
+\eta\lambda_n]^2\ .
\eeqano
For a given $n$, consider the function
\beq{gamma}
\Gamma_n(\tau,\xi,\eta) \equiv
\xi^2\ (-\tau^2+\lambda_n)^2+[\eta(-\tau^2+\lambda_n)+\tau]^2\ ,
\eeq
where $\tau \in \real$, $\xi+i \eta\in\Omega_{\sigma,B,\alpha}$. Clearly,
$$
\Gamma_n(2\pi\omega\cdot k,\xi,\eta)=|(\xi+i\eta)\lambda_{n,k}(\xi+i\eta)|^2\ ;
$$
since $\inf_{\tau\in\real} \Gamma_n(\tau,\xi,\eta)\leq \inf_{k\in\integer^d\backslash\{0\}} \Gamma_n(2\pi\omega\cdot k,\xi,\eta)$,
it suffices to bound from below $\inf_{\tau\in\real} \Gamma_n(\tau,\xi,\eta)$.

Let us start by considering the boundary $\{\xi=B\eta^\alpha\}$, namely we consider $\varepsilon$ as
\begin{equation}\label{parabola}
\varepsilon=B\eta^\alpha+i\eta
\end{equation}
with $\eta\in\real\backslash\{0\}$, $B>0$ large enough, say $B>B_0$ for some
$B_0\in\real_+$. Clearly, for every $n$ we have that
\[ \inf_{\varepsilon=B\eta^\alpha+i\eta}\ \inf_{k\in \integer^{d}\setminus\{0\}} |\varepsilon \lambda_{n,k}(\varepsilon)|^2 \geq
\inf_{\tau \in \real} |\Gamma_n(\tau,B\eta^\alpha,\eta)|\ .\]
We recall that {\bf H1} and {\bf H2} imply that $\inf_{n\geq 1}|\lambda_n|=\lambda_1>0$.

Obtaining lower bounds of \equ{gamma} for $\varepsilon$ of the form \eqref{parabola}
is very simple: indeed it is  the sum of two
non--negative terms, which vanish at very different places.
We analyze carefully each of the places where one of the terms vanishes. If none of the terms vanishes, the lower
bound is clear.

Let us define the three regions
\beqa{Ipm}
I_+&\equiv&[\sqrt{\lambda_n}-10^{-3}\sqrt{\lambda_n},\sqrt{\lambda_n}+10^{-3}\sqrt{\lambda_n}]\nonumber\\
I_-&\equiv&[-\sqrt{\lambda_n}-10^{-3}\sqrt{\lambda_n},-\sqrt{\lambda_n}+10^{-3}\sqrt{\lambda_n}]
\eeqa
and the complement of $I_+\cup I_-$. When $\tau$ is in such complement we have
\beq{L0}
\Gamma_n(\tau,\xi,\eta) \geq (B\eta^\alpha)^2(-\tau^2+\lambda_n)^2\geq
C_1^2(B\eta^\alpha)^2\lambda_n^2\geq C_1^2(B\eta^\alpha)^2\ \lambda_1^2
\eeq
for a suitable constant $C_1>0$. When $\tau\in I_+\cup I_-$ we have
\[
|\lambda_n-\tau^2|=|\sqrt{\lambda_n}-\tau|\ |\sqrt{\lambda_n}+\tau|\leq 10^{-3} \sqrt{\lambda_n}\ (2+10^{-3})\sqrt{\lambda_n}\ ,
\]
so that we obtain
\beqano
\Gamma_n(\tau,\xi,\eta) &\geq&[\eta(\lambda_n-\tau^2)+\tau]^2\nonumber\\
&\geq&\tau^2-2|\tau|\, |\eta|\, |\lambda_n-\tau^2|\nonumber\\
&\geq&\tau^2-2|\tau|\, |\eta|\, (2+10^{-3})\ 10^{-3}\,\lambda_n\ .
\eeqano
Since we have that
$|\tau|\leq (1+10^{-3}) \sqrt{\lambda_n}$ and $\tau^2\geq (1-10^{-3})^2 \lambda_n$,
we obtain for $\tau\in I_+\cup I_-$:
$$
\Gamma_n(\tau,\xi,\eta) \geq (1-10^{-3})^2\lambda_n-2(1+10^{-3})\ (2+10^{-3})\,10^{-3}\, \lambda_n^{3\over 2}\ |\eta|\ .
$$
Since we are considering $\lambda_1\leq \lambda_n\leq\lambda_K$, the following
inequality holds for all $\tau$:
\beq{L1}
\Gamma_n(\tau,\xi,\eta) \geq C_2^2 (B\eta^\alpha)^2
\eeq
for some constant $C_2>0$, provided $|\eta|$ is sufficiently small to satisfy the condition
$$
C_2^2 (B\eta^\alpha)^2+2(1+10^{-3})\ (2+10^{-3})\,10^{-3}\, \lambda_K^{3/2}\ |\eta|\leq (1-10^{-3})^2\, \lambda_1\ .
$$

We now consider the remaining parts of the boundary of the region $\Omega_{\sigma,B,\alpha}$, starting from the circle $|\varepsilon|^2=\sigma^2$ and
we begin from the case $|\lambda_n-\tau^2|<\delta$ for some positive $\delta$.
Then, for $\sigma$ sufficiently small and setting $\varepsilon=\xi+i \eta$, we have
\begin{equation}\label{L3}
\begin{aligned}
\Gamma_n(\tau,\xi,\eta)  &= \sigma^2(\tau^2-\lambda_n)^2+(1-2\eta \tau)(\tau^2-\lambda_n)+\lambda_n \\
&\geq \sigma^2(\tau^2-\lambda_n)^2+\frac{1}{2}(\tau^2-\lambda_n)+\lambda_n\geq \frac{\lambda_1}{2}\geq C_3 \xi^2
\end{aligned}
\end{equation}
for $C_3>0$, if $\sigma$ and $\delta$ are small enough with $0<\delta\leq \lambda_1$ and provided
$$
|\xi|\leq\sqrt{{{\lambda_1}\over {2C_3}}}\ .
$$

When $|\lambda_n-\tau^2|\geq\delta$, we have -as before- that the minimum
of $\Gamma_n(\tau,\xi,\eta)$ is reached for $|\lambda_n-\tau^2|=\delta$ and we
obtain the following bound:
\beq{L5}
\Gamma_n(\tau,\xi,\eta) = |\varepsilon|^2(\lambda_n-\tau^2)^2+\tau^2+2\tau \eta(\lambda_n-\tau^2)
=\xi^2\delta^2+(\eta(\lambda_n-\tau^2)+\tau)^2
\geq\xi^2\delta^2\ .
\eeq
Of course on the circle $|\varepsilon|^2=4\sigma^2$ we can reason
in the same way, possibly changing the constants.

\vskip.1in

Let us discuss now the case $n>K$; again we distinguish two cases.

For $\tau$ such that $|\lambda_n-\tau^2|<1$ we have $\tau^2> \lambda_K-1$; therefore for $|\eta|$
sufficiently small and for some $C_4>0$ we obtain the bound:
\beqa{L4}
\Gamma_n(\tau,\xi,\eta) &=&|\varepsilon|^2(\lambda_n-\tau^2)^2+\tau^2+2\tau \eta
(\lambda_n-\tau^2)\geq \tau^2-2|\tau||\eta|\geq {1\over 2}\tau^2 \nonumber\\
&\geq& {1\over 2}(\lambda_K-1)>\frac{\lambda_1}{2}
\geq C_4 \xi^2\ ,
\eeqa
provided $|\xi|$ is sufficiently small, namely
$$
|\xi|\leq\ \sqrt{{{\lambda_1}\over{2C_4}}}\ .
$$
Finally, for $|\lambda_n-\tau^2|\geq 1$, $n>K$, we have
\beq{L2}
\Gamma_n(\tau,\xi,\eta) \geq \xi^2|\lambda_n-\tau^2|^2 \geq \xi^2\ .
\eeq

Casting together the bounds \equ{L0}, \equ{L1}, \equ{L3}, \equ{L5}, \equ{L4}, \equ{L2}, there exists a constant $C_5>0$,
depending on $\lambda_1$, such that for every $n$,
\beq{L6}
\Gamma_n(\tau,\xi,\eta) \geq C_5 \ \max\{(B\eta^\alpha)^2,\xi^2\}=C_5 \xi^2\ ,
\eeq
since we are in the domain $\Omega_{\sigma,B,\alpha}\subseteq\{\xi+i\eta\in\complex:\ |\xi|\geq B|\eta|^\alpha\}$.
This concludes the bounds on the spectrum of $N_\varepsilon$.

\begin{remark}
$(i)$ The bounds providing the invertibility of $N_\varepsilon$ fail when $\varepsilon$ is on the imaginary axis.

$(ii)$ We do not assume that the spectrum is discrete; we could take the spectrum ranging
over any set of the real line.

$(iii)$ The discussion above does not depend explicitly on the boundary
conditions assumed for the PDE. However the boundary conditions enter through the
assumption {\bf H2}.

$(iv)$ The case in which
$$
\sup_{n\geq 1} \lambda_n =\Lambda<\infty
$$
is even simpler than the previous discussion, since in this case we can reason exactly as we did just
for the case $n\leq K$.
\end{remark}

To use the fixed point argument formulated in Section~\ref{sec:fixed}, we need the bound
$$
\left(\Gamma_n(\tau,\xi,\eta)\right)^{1\over 2}\ge C_6\sigma^2
$$
for some constant $C_6$. This bound will be used in \equ{NFP} below and,
in view of \equ{L6}, it amounts to requiring
$$
|\xi|\geq \tilde C_6\sigma^2
$$
for some constant $\tilde C_6>0$. This inequality is in turn implied by
$$
|B\eta^\alpha| \geq \tilde C_6\sigma^2\ ,
$$
which is possible only if
$$
|B\eta^\alpha|\geq \tilde C_6(B^2\eta^{2\alpha}+\eta^2)\ .
$$
Therefore, since $|\eta|<1$, we must have $\alpha\leq 2$; in conclusion, we take $\alpha=2$, being the best possible exponent,
thus leading to define the domain $\Omega_{\sigma,B}$ as in \equ{parabolicdomain}.

\subsection{Existence of the fixed point}\label{sec:fixed}
As we have discussed in Section \ref{sec:res1}, we can rewrite \eqref{rearranged} as a fixed point
equation, namely
\[
U(\theta, x)= N_\ep^{-1}(f(\theta,x)-\av{f}(x))-N_\ep^{-1}G(U)(\theta,x)\ ,
\]
where $U$ denotes a function of $\ep$ defined by $U_\ep=U_\ep(\theta,x)$.
In this way, we define an operator $\Tau$ acting on functions
analytic in $\ep$, taking values in $\A_{\rho, j, m}$, given by
\beq{operator}
\Tau(U)\equiv N_\ep^{-1}(f-\av{f})-N_\ep^{-1}G(U)\ .
\eeq
For a fixed $\ep$, we find a fixed point of $\Tau$ by considering a
domain $\mathcal{P} \subset \A_{\rho, j, m}$ with $\Tau(\mathcal{P})\subset \mathcal{P}$
on which $\Tau$ is a contraction.
Since we want to obtain analyticity in $\ep$, we reinterpret \equ{operator} as an operator
acting on a space of analytic functions in $\ep$ and we consider
the domain $\mathcal{\tilde P}$ in the space $\A_{\rho,j,m,\sigma,B}$
consisting of analytic functions of $\ep$ taking values in $\A_{\rho, j, m}$
with $\ep$ ranging on the domain $\Omega_{\sigma, B}$.

We endow $\A_{\rho, j, m,\sigma,B}$ with the supremum norm
\begin{equation}\label{supnorm}
\|U\|_{\rho,j,m,\sigma,B}\equiv\sup_{\ep\in\Omega_{\sigma, B}}\|U\|_{\rho, j, m}\ ,
\end{equation}
for which $\A_{\rho, j, m,\sigma,B}$ is a Banach Space. Moreover,
due to Proposition \ref{properties} of Appendix~\ref{app.spaces}, if $j > d$ and $m > \ell/2$, then
$\A_{\rho, j, m,\sigma,B}$ with the norm \eqref{supnorm}
is a Banach Algebra.

Notice that \eqref{L6} and the fact that $N_\varepsilon$ is diagonal implies that we can estimate
its norm in the domains $\Omega_{\sigma,B}$. We note that
the infimum is reached at the boundary of the domains, namely
\beq{NFP}
 \|N_\ep^{-1}\|_{\rho,j,m,\sigma,B} \leq C_7 B^{-1} \sigma^{-2}\sigma
\eeq
for some $C_7>0$, provided that $B$ is sufficiently large.
Using the Banach Algebra property
of $\A_{\rho, j, m,\sigma,B}$ and the fact that $D_U G(0)(\theta,x) =0$,
we note that the operator $\Tau_\ep$ is Lipschitz in a ball $\mathcal{B}_\alpha(0)
\subset \A_{\rho, j, m,\sigma,B}$ of
radius $\alpha>0$.
Indeed, using Proposition \ref{prop:composition} we have
that the Lipschitz constant of the composition with $G$
is bounded by a constant times $\alpha$. Thus, we have shown that
$$
\|\Tau(U)-\Tau(V)\|_{\rho, j, m, \sigma,B}\le C_7 B^{-1} \sigma^{-1}
\alpha \|U-V\|_{\rho, j, m,\sigma,B}\ .
$$
We continue as in \cite{CallejaCL13} by showing that $\Tau$ is
a contraction in a ball
centered around the approximate solution that gets mapped into
itself. First, we notice that the
approximate solution $U^M=U^M(\theta,x)$
(see Definition~\ref{def:approx}) satisfies
$$
\|U^M\|_{\rho,j,m,\sigma,B}\leq C_{8}\sigma
$$
for some $C_{8}>0$.

We fix $\alpha_0 > 0$  which will be the radius of a ball around
zero in $\A_ {\rho, j,m,\sigma,B}$, so that we will
take the constants corresponding to this
ball.  We will refer to this ball as the \sl ambient ball. \rm

Our next goal will be to identify balls around the approximate solutions
such that the operator $\Tau$ maps them into themselves and
is a contraction. The following discussion is very similar to what is done in
\cite{CallejaCL13}.

Consider a ball $\B_\beta(U^M)$ of radius $\beta$ around $U^M$.
We will impose several conditions on $\beta$ that ensure that the ball is mapped into itself by $\Tau$ and
that $\Tau$ is a contraction.

The ball $\B_\beta(U^M)$
 is contained in the ambient ball, $\mathcal{B}_{\alpha_0}(0)\in\A_ {\rho, j,m,\sigma,B}$, provided that
\beq{Cond2}
C_{8}\sigma+\beta\leq \alpha_0\ .
\eeq
Hence, we will assume \equ{Cond2} to ensure we can use the constants of the operator in the ambient ball.

The operator  $\Tau$ is a contraction on  $\B_\beta(U^M)$
provided that
\beq{Cond1}
C_7(C_{8}\sigma+\beta)B^{-1}\sigma^{-1}<1\ .
\eeq

Moreover, we have that the approximate solution satisfies the inequality
\[
\|\Tau(U^M)-U^M\|_{\rho,j,m,\sigma,B}\leq C_{9}\sigma^3 B^{-1}\sigma^{-1}
\]
for some constant $C_{9}>0$, since $U^M$ is a solution at least to $O(\varepsilon^3)$
as in Theorem~\ref{thm:main} or \ref{convergence}.
The ball $\B_\beta(U^M)$ is mapped into itself, whenever
\beq{Cond3}
C_7(C_{8}\sigma+\beta)B^{-1}\sigma^{-1}\beta+C_{9}B^{-1}\sigma^2\leq\beta\ .
\eeq
Notice that to fulfill \equ{Cond2}, \equ{Cond1}, \equ{Cond3}, we are allowed to choose $\beta$. Namely, we want to show that
for some $B$ large enough and for all $\sigma$ sufficiently small, say $\sigma\leq\sigma^\ast(B)$, we can find
$\beta>0$ such that the three conditions \equ{Cond2}, \equ{Cond1}, \equ{Cond3} are satisfied.

It is natural to choose
$$
\beta=100\ \sigma
$$
and, then, \equ{Cond2}, \equ{Cond1}, \equ{Cond3} are implied by
\beq{Cond1p}
(C_8+100)\sigma \leq \alpha_0
\eeq
\beq{Cond2p}
C_7(C_8+100)B^{-1} < 1
\eeq
\beq{Cond3p}
100 C_7 (C_8+100)B^{-1}+C_9 B^{-1} \sigma \leq 100\ .
\eeq
We see that we can choose $B$ large enough so that \equ{Cond2p} is satisfied and, then, \equ{Cond1p}, \equ{Cond3p}
are satisfied for $\sigma$ small enough.

In conclusion, we obtain that $\Tau$ admits a fixed point in the domain $\tilde\P$, provided $\sigma$ and $B$ are
suitably chosen. This fixed point will be a function analytic in $\Omega_{\sigma,B}$.

As a corollary, the solution is locally unique, namely we have the following result.

\begin{cor}\label{cor:uniqueness}
For a fixed $\varepsilon\in\Omega_{\sigma,B}$ with $\sigma$, $B$ such that \equ{Cond2}, \equ{Cond1}, \equ{Cond3} are satisfied,
let $U^M$ be an approximate solution. Then, for any $\theta\in\torus^d$, we have that
\beq{UL}
\lim_{n\to\infty} \Tau^n\ U^M(\theta)=U(\theta)\ .
\eeq
In particular, the convergence in \equ{UL} is uniform for $\varepsilon\in\Omega_{\sigma,B}$ with
$\Omega_{\sigma,B}$ as in \equ{parabolicdomain}, since $\|U^M-U\|\leq C\sigma^M$ for a
positive constant $C>0$, which implies that the solution is analytic for $\varepsilon\in\Omega_{\sigma,B}$.
\end{cor}

\section{Modifications of the proof in Section \ref{sec:proof} to models A', B, B'}\label{sec:others}

In this Section we consider models A', B, B', providing the necessary modifications to the proof developed for model A. In particular, we concentrate on the extension of
Theorem~\ref{thm:allorders} to construct an approximate solution and on the bound of the eigenvalues in a
suitable domain $\Omega_{\sigma,B}$, as it was done for model A in Section~\ref{sec:bounds}. The other parts of the proof can be extended
to models A', B, B', trivially.
The existence of the order zero solution is considered in Appendix \ref{sec:c0A}, and
Sections \ref{sec:c0B}, \ref{sec:c0Bp}, respectively.

\subsection{Model A'}\label{sec:othersC}
For model A' we look for a response solution of the form \eqref{response} and
define the operators $N_\varepsilon$, $\Lop$ and $G$ as
\beqa{modAp}
N_\varepsilon U(\theta,x)&\equiv& [(\omega\cdot\nabla_\theta)^2+{1\over\varepsilon}(\omega\cdot\nabla_\theta)\Delta_x+\Lop]U(\theta,x)\nonumber\\
\Lop U(\theta,x)&\equiv& -\Delta_x U(\theta,x)+h'(c_0(x),x)U(\theta,x)\nonumber\\
{G(U)}(\theta,x)&\equiv& h(c_0(x)+U(\theta,x),x)-h(c_0(x),x)-h'(c_0(x),x)U(\theta,x)\ .
\eeqa
The equation \equ{hullPDEC} is equivalent to the equation
$$
N_\varepsilon U_\varepsilon(\theta,x)+G(U_\varepsilon)(\theta,x)=f(\theta,x)-\langle f\rangle(x)\ ,
$$
while the function $c_0$ must satisfy \eqref{c0}

\vskip.1in

To construct an approximate solution, let us again write formally $U_\varepsilon$ as $U_\varepsilon=\sum_{j=1}^\infty \varepsilon^j U_j$.
Then, after solving equation \equ{c0} for $c_0$ (see Appendix~\ref{sec:c0A}), at the first order in $\varepsilon$ we need to solve
$$
(\omega\cdot\nabla_\theta)\Delta_x U_1(\theta,x)=f(\theta,x)-\langle f\rangle (x)\ ,
$$
which yields the non-average part of $U_1$ using {\bf H2'} and the non-resonance condition on $\omega$.
At the second order in $\varepsilon$ we obtain the equation
$$
(\omega\cdot\nabla_\theta)\Delta_x U_2(\theta,x)=-[(\omega\cdot\nabla_\theta)^2 + \Lop]U_1(\theta,x)\ ,
$$
from which we first deduce that the average $\langle U_1\rangle$ should be zero by
imposing that the right hand side has zero average and noting that
$\Lop$ is invertible. Then, we determine
the non-average part of $U_2$ by solving the remaining equation.

At the order $N\geq 3$ we obtain the equation
\beq{orderN-Ap}
(\omega\cdot\nabla_\theta)\Delta_x U_N(\theta,x) = -[(\omega\cdot\nabla_\theta)^2 + \Lop]U_{N-1}(\theta,x)
 + S_N(c_0(x),U_1(\theta,x),...,U_{N-2}(\theta,x))
\eeq
for a suitable function $S_N$, depending on $c_0$ and on the functions $U_j$ with $j<N-1$;
by imposing that the average of the right hand side of \eqref{orderN-Ap}
is zero we get
$$
\Lop (\langle U_{N-1} \rangle)  = \langle S_N(c_0,U_1...,U_{N-2}) \rangle\ .
$$
After fixing the average of $U_{N-1}$, we obtain the non-average part of $U_N$ by solving
equation \eqref{orderN-Ap}.

\vskip.1in

To conclude the proof for model A', we proceed to estimate the eigenvalues of the
operator $N_\varepsilon$ in a way similar to that of model A.

Indeed, let us write $\varepsilon=\xi+i\eta$ with $\varepsilon$ belonging to the domain
$\Omega_{\sigma,B}$ defined in \equ{parabolicdomain}; denoting by $\lambda_n^\Delta$ the eigenvalues
associated to $-\Delta_x$, we have:
\beqano
|\varepsilon\lambda_{n,k}|^2&=&|-\varepsilon (2 \pi \omega \cdot k)^2+i(2 \pi \omega \cdot k) \lambda_n^\Delta+\varepsilon\lambda_n|^2\nonumber\\
&=&\xi^2((2 \pi \omega \cdot k)^2-\lambda_n)^2+[-\eta (2 \pi \omega \cdot k)^2+(2 \pi \omega \cdot k)\lambda_n^\Delta+\eta\lambda_n]^2\ .
\eeqano
As for model A, we introduce an auxiliary function $\Gamma_n(\tau,\xi,\eta)$ to obtain bounds on the eigenvalues
of $N_\ep$:
\beq{gammaC}
\Gamma_n(\tau,\xi,\eta) \equiv \xi^2(\tau^2-\lambda_n)^2+[\eta(-\tau^2+\lambda_n)+\tau\lambda_n^\Delta]^2\ .
\eeq
Again we fix $K\in\integer$ such that $\lambda_K-1\geq\lambda_1$ and consider first
the case $n\leq K$.
Define the regions $I_-$ and $I_+$ as in \equ{Ipm}. In the region $(I_-\cup I_+)^c$ we obtain
$$
\Gamma_n(\tau,\xi,\eta)\geq \xi^2(\tau^2-\lambda_n)^2\geq C_{10} \xi^2\lambda_n^2\geq C_{10} \xi^2\lambda_1^2
$$
for a suitable constant $C_{10}>0.$

Within the region $I_-\cup I_+$ we  have
\beqano
\Gamma_n(\tau,\xi,\eta)&\geq&(-\eta\tau^2+\tau\lambda_n^\Delta+\eta\lambda_n)^2\geq \tau^2(\lambda_n^\Delta)^2-2|\eta|\ |\tau|\ |\lambda_n-\tau^2|\lambda_n^\Delta\nonumber\\
&\geq&\tau^2(\lambda_n^\Delta)^2-2\lambda_n^\Delta|\eta|\ |\tau|(2+10^{-3})\,10^{-3}\,\lambda_n\nonumber\\
&\geq&(1-10^{-3})^2\lambda_n(\lambda_n^\Delta)^2-2\lambda_n^\Delta(1+10^{-3})\lambda_n^{3/2}(2+10^{-3})\,10^{-3}\,|\eta|\geq C_{11}\xi^2\ ,
\eeqano
for some constant $C_{11}>0$, if
\beq{estC}
C_{11}\xi^2+2\lambda_K^\Delta(1+10^{-3})(2+10^{-3})\,10^{-3}\,\lambda_K^{3/2}|\eta|\leq(1-10^{-3})^2
(\lambda^\Delta_1)^2\lambda_1\ .
\eeq

In the case $n>K$, we recall the expression of $\Gamma_n(\tau,\xi,\eta)$ in \equ{gammaC}.
Then, we consider the subcase $|\lambda_n-\tau^2|<1$, which provides $\tau^2> \lambda_K-1$, so that one obtains
for $|\eta|$ sufficiently small:
\beqano
\Gamma_n(\tau,\xi,\eta)&=&|\varepsilon|^2(\lambda_n-\tau^2)^2+(\lambda_n^\Delta)^2\tau^2+2\tau\ \lambda_n^\Delta\ \eta
\ (\lambda_n-\tau^2)\nonumber\\
&\geq&(\lambda_n^\Delta)^2\tau^2-2\lambda_n^\Delta|\eta|\ |\tau|\nonumber\\
&\geq&{1\over 2}(\lambda_n^\Delta)^2\tau^2\geq {1\over 2}(\lambda_1^\Delta)^2\
(\lambda_K-1)>(\lambda_1^\Delta)^2\frac{\lambda_1}{2}\ ,
\eeqano
provided
$$
|\eta|\leq{1\over 4}\lambda_1^\Delta\sqrt{\lambda_K-1}\ .
$$

In the case $|\lambda_n-\tau^2|\geq 1$, recalling \equ{gammaC} we obtain
$$
\Gamma_n(\tau,\xi,\eta)\geq \xi^2 |\lambda_n-\tau^2|^2 \geq \xi^2\ .
$$

Concerning the boundaries $|\varepsilon|=\sigma$ and $|\varepsilon|=2\sigma$, again we can
reason as  for model A.\\

Note that the assumption {\bf H2'} is crucial in order to get the bound; if $\Delta_x$ is the standard Laplace-Beltrami
operator, the invertibility
of the operator $N_\varepsilon$ is guaranteed only for D boundary conditions, because {\bf H2'}
is violated for N, P boundary conditions.

This concludes the discussion of the invertibility of the operator $N_\varepsilon$ for model A'.
The existence of a fixed point can be done in full analogy to model A; see Section~\ref{sec:fixed}.

\subsection{Model B}\label{sec:othersB}

To construct an approximate solution, let us write formally
$U_\ep(\theta,x)=U_0(\theta,x)+\sum_{j=1}^\infty \varepsilon^j U_j(\theta,x)$.
Given the zeroth order solution (see \equ{c0U0}) as in Section~\ref{sec:c0B}, we proceed to determine the higher order terms, matching
powers of the formal series expansion in the equation
\beq{hullB}
\varepsilon^2(\omega\cdot\nabla_\theta)^2 U_\ep(\theta,x)+(\omega\cdot\nabla_\theta) U_\ep(\theta,x)-\Delta_x U_\ep(\theta,x)+
h(U_\ep(\theta,x),x)=f(\theta,x)\ .
\eeq
At first order in $\varepsilon$, we get the equation
\beq{U1B}
(\omega\cdot\nabla_\theta) U_1(\theta,x)-\Delta_x U_1(\theta,x)+
h'(U_0(\theta,x),x) U_1(\theta,x)=0\ ,
\eeq
which can be used to determine $U_1$. Indeed, writing \equ{U1B} as
$$
[(\omega\cdot\nabla_\theta)-\Delta_x+h'(U_0(\theta,x),x)]\ U_1(\theta,x)=0\ ,
$$
we may fix $U_1\equiv 0$.
\begin{remark}\label{rmk:neumann}
Note that if $U_0$ is contained in a ball around the origin in $\A_{\rho,j,m}$ with small enough radius as in
Proposition~\ref{pro:zerob}, then the operator
\beq{lamtilde}
\tilde \Gamma =(\omega\cdot\nabla_\theta)-\Delta_x+h'(U_0(\theta,x),x)
\eeq
is invertible. Indeed, if we write it as the sum of the invertible operator \eqref{Gamma},
introduced in Section \ref{sec:c0B} plus the multiplication operator $T$ defined as
$$T\phi = [h'(U_0(\theta,x),x) - h'(0,x)]\phi\ ,$$
which is small when $U_0$ is in a small ball, we obtain the inverse
of $\tilde \Gamma = \Gamma + T$
by a Neumann series argument.
\end{remark}

At the generic order $N\ge2$, we obtain the equation:
\beq{UNB}
\tilde \Gamma \ (U_N(\theta,x))=-(\omega\cdot\nabla_\theta)^2 U_{N-2}(\theta,x)
+S_N(U_0(\theta,x),...,U_{N-1}(\theta,x))\ ,
\eeq
where $S_N$ is a known function of the $U_j$'s with $j<N$. If the operator $\tilde\Gamma$
is invertible, we can determine $U_N$ uniquely.

Now we assume that it is possible to solve the zeroth order equation \equ{c0U0}
(some sufficient conditions have been presented is Section~\ref{sec:c0B})
as well as to solve the recursive equations
\equ{U1B}, \equ{UNB} by taking $U_0$ in a small ball
around the origin in the $\A_{\rho,j,m}$ norm.
We proceed to study the conditions under which \equ{hullB} can be solved.

We start by introducing the  operator
\beq{Lambaeps}
\Lambda_\varepsilon=\varepsilon^2(\omega\cdot\nabla_\theta)^2+(\omega\cdot\nabla_\theta)-\Delta_x
+h'(0,x)\ .
\eeq
If this operator is invertible, by the same argument as in Remark \ref{rmk:neumann}
the operator
\beq{Lambaeps}
\tilde \Lambda_\varepsilon=\varepsilon^2(\omega\cdot\nabla_\theta)^2+(\omega\cdot\nabla_\theta)-\Delta_x
+h'(U_0(\theta,x),x)
\eeq
is invertible whenever $\|U_0\|_{\rho,j,m}$ is sufficiently small.
Now, if we write equation \eqref{hullB} as
\beq{hullBlam}
\tilde \Lambda_\varepsilon U_\ep(\theta,x) + H(U_\ep)(\theta,x) = f(\theta,x)
\eeq
where we write $U_\ep = U_0 + \tilde U_\varepsilon$ with $\tilde U_\varepsilon=\sum_{j=1}^\infty \varepsilon^j U_j(\theta,x)$ and
$$
H(U_\ep)(\theta,x)= h(U_\ep(\theta,x),x) - h'(U_0(\theta, x),x) U_\varepsilon(\theta,x)\ ,
$$
we are led to solve the equation
$$
U_\ep(\theta,x)=-\tilde \Lambda_\varepsilon^{-1}[H(U_\ep)(\theta,x) - f(\theta, x)]\ .
$$
Let us define the operator $\Tau$ acting on a function $U=U(\theta,x)$ by
\beq{TaumodelB}
\Tau [U](\theta,x)\equiv -\tilde \Lambda_\varepsilon^{-1}[H(U)(\theta,x) - f(\theta, x)]\ .
\eeq
Using Proposition~\ref{prop:composition} of Appendix~\ref{app.spaces}, we can show that
for $U, V \in \A_{\rho, j, m}$, then $\Tau$ satisfies the inequality:
$$
\|\Tau(U)-\Tau(V)\|_{\rho,j,m}=\|\tilde \Lambda_\varepsilon ^{-1}\ (H(U))
-\tilde \Lambda_\varepsilon ^{-1}\ (H(V))\|_{\rho,j,m}
\leq C\alpha_0\ \|\tilde \Lambda_\varepsilon ^{-1}\|_{\rho,j,m}\ \|U - V\|_{\rho,j,m}\ ,
$$
since the Lipschitz constant of the composition with $H$ is bounded by a constant times $\alpha_0$.

As in the case of model A, to check that $\Tau$ maps a small enough ball around an approximate
solution $U^M(\theta,x)$ into itself and it is a contraction, we need to investigate
 the domain on which $\Lambda_\varepsilon$ can be inverted with ``good bounds''.

The multiplier $\lambda_{n,k,\ep}$ associated to $\tilde\Lambda_\varepsilon$ is given by
$$
\lambda_{n,k,\ep}\equiv \varepsilon^2 (2\pi i\omega\cdot k)^2+2\pi i\omega\cdot k+\lambda_n\ ,
$$
where the eigenvalues $\lambda_n$ of $\Lop\equiv-\Delta_x+h'(U_0(\theta,x),x)$ satisfy
{\bf H1}-{\bf H2} with the assumption that $\|U_0\|_{\rho,j,m}$ is small.
For a given $n$, we consider the function
$$
\tilde\Gamma_n(\tau,\ep)\equiv -\varepsilon^2\tau^2+i \tau+\lambda_n
$$
for $\tau\in\real$. We are interested to evaluate the quantity $\inf |\tilde\Gamma_n(\tau,\ep)|$.
This function can be easily analyzed geometrically, since the part corresponding to
$-\varepsilon^2\tau^2+i \tau$ is a parabola. The infimum is generated by considering the minimum
distance of the parabola from the quantity $\lambda_n$ (which is a real number).

For $\varepsilon=0$ the parabola coincides with the vertical axis, so that if $\lambda_n\not=0$,
the distance is always positive.
Indeed, the parabola $-\varepsilon^2\tau^2+i \tau$ passes through the origin and it is tangent there
to $i\real$. The axis of the parabola coincides with $-\varepsilon^2$.

We assume that
\beq{R}
{\rm Re}(-\varepsilon^2)\geq \delta> 0\ ,
\eeq
setting $\ep=B\eta^2+i\eta$, then \equ{R}
amounts to requiring that $\eta^2-B^2\eta^4\geq\delta> 0$, which is satisfied for $\eta$
sufficiently small. Finally, we obtain the estimate:
$$
|{\rm Re}(-\varepsilon^2\tau^2+i\tau+\lambda_n)|\geq \delta\tau^2+\lambda_n\geq\lambda_n\ ,
$$
which ensures that the spectrum of $\tilde\Lambda_\ep$ is away from zero due to {\bf H1}-{\bf H2}.
Therefore, we infer that the operator $\Gamma_n(\tau,\ep)$ is invertible and
we get uniform bounds within the domain
$$
\Omega_\delta\equiv\{\ep=\xi+i\eta:\ {\rm Re}(-\ep^2)\geq\delta\}
$$
for $\delta>0$.

Let us conclude by considering the case of $\varepsilon$ real
which is not covered by \equ{R}, say $\varepsilon=\xi$ with $\xi\in\real$ as in \equ{omd}
with $\delta$ small enough. Then, setting
$$
\Gamma_n(\tau,\xi)\equiv|-\xi^2\tau^2+i\tau+\lambda_n|^2=(\lambda_n-\xi^2\tau^2)^2+\tau^2\ ,
$$
it follows that
$$
\frac{d}{d\tau}\Gamma_n(\tau,\xi)=2\tau (2\xi^4\tau^2-2\lambda_n \xi^2+1)\ .
$$

We have two cases:

\noindent
{\bf case 1.} $2\lambda_n \xi^2\leq 1$, so that the minimum is attained at $\tau=0$ and one has
$\Gamma_n(0,\xi)=\lambda_n^2\geq \lambda_1^2$;

\noindent
{\bf case 2.} $2\lambda_n\xi^2> 1$ and hence the minimum is attained at $\tau_{\pm}=\pm\sqrt{(2\lambda_n\xi^2-1)/(2\xi^4)}$
($\Gamma_n(\tau_\pm,\xi)$ are equal for parity reasons) and one has
$$
\Gamma_n(\tau_\pm,\xi) =\frac{1}{4\xi^4}+\frac{\lambda_n}{\xi^2}-\frac{1}{2\xi^4}=\frac{\lambda_n}{\xi^2}-\frac{1}{4\xi^4}
\geq \frac{1}{4\xi^4}\geq1\,,
$$
for $\xi$ small enough.

Summarizing, for $\ep$ real  we get
\beq{es2}
\Gamma_n(\tau,\xi) \geq \min\{\lambda_1^2,1\}\ .
\eeq

\subsection{Model B'}\label{sec:othersBprime}

We write the solution as $U_\ep(\theta,x)=U_0(\theta,x)+\tilde U_\ep(\theta,x)$, where
$\tilde U_\ep\equiv \sum_{j=1}^\infty \ep^j U_j(\theta,x)$.
The zeroth order solution has been already discussed in Section~\ref{sec:c0Bp}. With respect to model B,
the only modification is that here we do not need the assumption $h(0)=0$. To determine the higher order
terms $U_j$, we start by considering the equation
\beq{Upeq}
\varepsilon^2(\omega\cdot\nabla_\theta)^2U_\ep(\theta,x)+(\omega\cdot\nabla_\theta) U_\ep(\theta,x)
-\Delta_x U_\ep(\theta,x)+\varepsilon h(U_\ep(\theta,x),x)=f(\theta,x)\ .
\eeq
Inserting the series expansion for $\tilde U_\ep$ into \equ{Upeq} and matching the same powers of $\varepsilon$, we get the equations for
the functions $U_j$, $j\geq 1$.

At the first order in $\varepsilon$ we obtain the equation:
\beq{Up1}
(\omega\cdot\nabla_\theta) U_1(\theta,x)-\Delta_x U_1(\theta,x)=-h(U_0(\theta,x),x)\ .
\eeq
Let $\Lambda\equiv(\omega\cdot\nabla_\theta)-\Delta_x$; then \equ{Up1} can be rewritten as
\beq{Up1new}
\Lambda U_1(\theta,x)=-h(U_0(\theta,x),x)\ ,
\eeq
and note that the right hand side is a known function, once we solved the zeroth order equation.

At the order $N\geq 2$ we get a recursive equation of the form
\beq{UpN}
\Lambda U_N(\theta,x)=S_N(U_0(\theta,x),U_1(\theta,x),...,U_{N-1}(\theta,x))\ ,
\eeq
for a function $S_N$ depending on the terms $U_j$, $0\leq j<N$, which are assumed to be determined
at the previous steps.

Both equations \equ{Up1new} and \equ{UpN} can be solved, provided the operator $\Lambda$ is boundedly invertible
in the spaces $\A_{\rho,j,m}$.
This requirement is satisfied under the conditions {\bf H1'}-{\bf H2'} on the eigenvalues
of $-\Delta_x$ appearing in $\Lambda$.

After solving the zeroth order equation as well as \equ{UpN}
up to a finite order $N$, we
consider the formulation of \equ{Upeq} as a fixed point
equation and establish the existence of solutions.

Let us define the operator $\Lambda_\varepsilon$ as
\beq{Lambaepsp}
\tilde \Lambda_\varepsilon=\varepsilon^2(\omega\cdot\nabla_\theta)^2+(\omega\cdot\nabla_\theta)-\Delta_x\ ;
\eeq
then, equation \equ{Upeq} can be written as
$$
\tilde \Lambda_\varepsilon U_\ep(\theta,x)+H_\ep(U_\ep)(\theta,x)=f(\theta,x)\ ,
$$
where $H_\ep$ is defined as
$$
H_\varepsilon(U_\ep)(\theta,x)=\varepsilon h(U_\ep(\theta,x),x)\ .
$$
We need to solve the fixed point equation
\beq{FPBp}
U_\ep(\theta,x)=-\tilde \Lambda_\varepsilon^{-1} [H_\varepsilon(U_\ep)(\theta,x)-f(\theta,x)]\ .
\eeq
The invertibility of the operator $\tilde \Lambda_\varepsilon$ in the space $\A_{\rho,j,m}$
has been already discussed for model B and we
conclude that condition \equ{R}
together with {\bf H1'}-{\bf H2'} ensure that the
spectrum of $\tilde \Lambda_\varepsilon$ is bounded away from zero,
if $U_0$ is in a sufficiently small ball around the origin, as required in Proposition~\ref{pro:zerob}.
Comparing \equ{TaumodelB} for model B and \equ{FPBp}, we conclude that we can reason as for model B
to apply the contraction mapping argument.

\section{Optimality of the results}\label{sec:optimal}
The domains of analyticity for response solutions established in Theorem~\ref{thm:main} are not optimal.
Clearly, many details of the argument can be optimized
and it is quite possible that one can use better fixed point theorems or better arguments.

Nevertheless, we want to argue in this section that the results presented cannot be
improved very dramatically and are qualitatively optimal.

We will present rigorous results (Theorem~\ref{thm:optimal})
and heuristic arguments  (Conjectures \ref{conj:optimal1}, \ref{conj:optimal2}, \ref{conj:optimal3})
that indicate that the results obtained are
qualitatively optimal and quantitatively almost optimal.

In particular, we believe that the domains of analyticity of the response
solutions for models B, B' do not contain sectors with aperture bigger than $\pi/2$ for generic perturbations.
Note that $\pi/2$ is precisely the
critical aperture of the Phragm\'en-Lindelof theorem (\cite{PhragmenL08,SaksZ65}), which makes the function theoretic properties
of the perturbative functions very tantalizing. Of course, this has deep consequences for the properties of
the asymptotic expansions and how to recover the function from the computed asymptotic expansion
(note that not even the uniqueness of the asymptotic expansion is clear for functions in these domains).

The argument presented in this section is very general and it applies
to many problems that can be reduced to a fixed point problem with parameters
and which satisfy some mild conditions on analyticity and compactness.
In particular, it applies to the treatment of the varactor problem
carried out  in \cite{CallejaCL13}, but we will not formulate here the precise
results in this case, even if they have less technicalities than those used in this work.

The argument we present here
goes by contradiction.  We show rigorously (see Theorem~\ref{thm:optimal}) that, if  there is a family of solutions $u_\ep$ whose domain
includes a \emph{resonance} (see Definition~\ref{def:resonance}), if we embed the problem in a two parameter family of
problems, then for most families we do not have a two parameter family of solutions. The second rigorous result (Lemma~\ref{twovar})
strengthens a bit the previous one by showing that if we had solutions for all the one-parameter families in a neighborhood
(in the space of analytic families), we could find
an analytic two-parameter family. The conclusion of the two results above is that it will be very unlikely to find a family of problems,
so that the domain of the response function includes  a resonance.
That is, if we  find a family whose solutions include a resonance,  we can find arbitrarily small perturbations whose  response solutions have
a domain that does not include the resonance.

By examining the argument carefully, and by proposing alternative points of view,  we speculate -- but we do not prove it rigorously --
that this argument applies to all resonances simultaneously. This leads to Conjecture~\ref{conj:optimal1}.

Of course, since our contradictions are obtained by constructing perturbations which cannot be continued, if we consider a
restricted class of models, one has to wonder whether the perturbations can be constructed in this class.

Similar lines of argumentation have appeared in the literature. Notably, we have been inspired by the use of
uniform integrability in \cite{Poincare92} to obtain insights on the problem of integrability.

\subsection{Statement of rigorous results on optimality}
The key to the arguments in this section  is the concept of resonance for a parameter family of solutions.

\begin{definition} \label{def:resonance}
Let $\O_\ep$ be an analytic family of bounded  operators from
a Banach space to itself.
We say that $\ep_0$ is a resonant value for the family $\O_\ep$, whenever
the operator $\O_{\ep_0} $ has a zero eigenvalue.

We say that the resonance is isolated if for all $0 < | \ep - \ep_0| \ll 1$,
we have that  $\O_\ep$ is invertible. Note that the  arguments  presented here
do not require that the resonance is
isolated.
\end{definition}

In our applications to non-linear problems, say $\F_\ep(U_\ep) = 0$, we will take
as the linear operators $\O_\ep$, the derivatives at the solution
$\O_\ep = D\F_\ep( U_\ep)$.

Of course for a general family of operators there are other alternatives
between having an eigenvalue zero and being invertible (e.g., having continuous
spectrum, having residual spectrum, etc.). In our case, the operators have a
spectrum which is the closure of the set of the eigenvalues.

The important fact about resonances is that if $\ep_0$ is a resonant value,
the range of the  operator  $\O_{\ep_0}$ has codimension at least $1$.

To prove our results  it will be useful to introduce a two-parameter family, say $\F_{\ep,\mu}$,
so that it will be easier to compute obstructions generated by resonances. Precisely,
our result is based on the following arguments:

\begin{itemize}
\item[$(i)$] we will show (see part $(b)$ of Theorem~\ref{thm:optimal})
that given a two-parameter family of problems,
$\F_{\ep,\mu}(U_{\ep,\mu} )= 0$, such that  $D\F_{\ep,0}(U_{\ep,0})$
is resonant, we cannot expect to obtain solutions analytic in $\mu$ near $\mu=0$
(we call this phenomenon \emph{``automatic analyticity''});
\item[$(ii)$] we will show that if there is a perturbative solution for
every one-parameter family, there has to be a jointly analytic solution in two
parameters (see Lemma~\ref{twovar});
\item[$(iii)$] the consequence of these two results
is that it is impossible that there is a solution that drives through
the resonances for every one-parameter family (see part $(a)$ of Theorem~\ref{thm:optimal}).
\end{itemize}

We will work mainly with the equation \equ{rearranged},
but - as anticipated before - we extend it adding a parameter for the nonlinearity.
In particular, we rewrite \equ{rearranged} as
\beq{equ:rearranged2}
\F_{\ep, \mu}(U_{\ep, \mu}) = N_{\ep} U_{\ep,\mu }  + A_\mu(U_{\ep, \mu})  = 0\ ,
\eeq
where $A_\mu(U_{\ep, \mu})\equiv G_\mu(U_{\ep, \mu})-f+\langle f\rangle$ and
$G_\mu$ is any smooth function of $\mu$ such that $G_0=G$.

We define the family of operators $\O_\ep$ as
\beq{Odefined}
\O_\ep \equiv D\F_{\ep,0}(U_{\ep, 0})= N_{\ep} + A'_0(U_{\ep, 0})\ .
\eeq
We will argue that if $\O_\ep$ has a resonance at $\ep = \ep_0$, it is very difficult
to have a family $A_\mu$ that allows us to have $U_{\ep_0, \mu}$
analytic in $\mu$ and which solves $\F_{\ep_0, \mu}(U_{\ep_0,\mu}) = 0$.

To make all this  precise, we endow the space of analytic families of
linear operators with the topology of  the supremum of the norm in a complex domain,
so that it is a Banach space.
We will always consider the domains in  $\mu$ to be a ball around
$\mu = 0$. The domains in $\ep$ could be either a ball around
$\ep = 0$ for the perturbative expansions or a ball around
$\ep = \ep_0$, where $\ep_0$ is a resonance. When dealing with functions
of two variables, we consider domains which are the product.

The key to the argument is to show that if there are analytic
solutions, the family $\F_{\ep_0, \mu}$ has to satisfy constraints
and that generic families violate them. Of course, if
one considers specific models in \eqref{equ:rearranged2},
it could in principle happen that the family automatically
satisfies the constraint. We will however show
that this does not happen in general and that, even in specific models
for \eqref{equ:rearranged2}, it is unlikely that one can make deformations
 satisfying the constraints imposed by the existence of
analytic solutions.

\begin{theorem}\label{thm:optimal}
Let  $\F_{\ep}$ be an analytic family of analytic operators from a Banach space to itself. Assume that there is a family $U_\ep$
defined  in a  domain of analyticity including
$\ep_0$, such that $\F_\ep(U_\ep)  = 0$ and that  $\ep_0$ is a resonant value for the  $D\F_\ep(U_\ep)$ family.

Consider an arbitrary small ball $B \subset \complex $ centered at $\ep_0$
and define $\A_B$ the space of analytic families of operators defined in
the ball endowed with the supremum topology.

Then, we have the following results.

$(a)$ In any sufficiently small ball of $\A_B$ centered at $\F_{\ep_0}$,
we can find a family of operators $\tilde \F_\ep$ such that the domain of
analyticity of the solution of $\tilde \F_\ep(U_\ep)=0$ does not include $\ep_0$.

$(b)$ For restricted two-parameter families $\F_{\ep,\mu}$ of the form \eqref{equ:rearranged2}, we
have the same result. Namely, if there is an analytic family $U_{\ep,0}$ satisfying
$\F_{\ep,0}(U_{\ep, 0}) = 0$ and $\ep_0$ is a resonant  value for
$D\F_{\ep,0}(U_{\ep, 0})$, then for
an open and dense set of families $A_\mu$, we can find  arbitrary small values
$\mu$, such that the family $\F_{\ep, \mu}$ does
not admit a solution $U_{\ep,\mu}$ which is analytic near $\mu=0$.
\end{theorem}

\subsection{Proof of Theorem~\ref{thm:optimal}}
The first element in the proof of Theorem~\ref{thm:optimal} is the
following elementary lemma showing that if one has analytic
solutions for all equations, then they have to be analytic in a second
parameter. Afterwards, we will identify obstructions for analyticity
in two parameters near a resonance (this obstruction is very similar
to Poincar\'e's obstructions to uniform integrability
(\cite[\S  81]{Poincare92}, see \cite{Llave96} for a reexamination of
\cite{Poincare92} with modern techniques and for a converse of the results of \cite{Poincare92}).

\begin{lemma}\label{twovar}
Consider a family of equations $\F_\ep(U) = 0$, where $\F_\ep$ is an
analytic family of nonlinear operators.  Endow the space of analytic operators with the supremum topology.

Assume that for all $\G_\ep$ in a neighborhood of $\F_\ep$ in the space of
analytic functions there is an analytic
solution $U_\ep$, which is locally unique. Then, for every two-parameter family
$\F_{\ep, \mu}$, such that $\F_{\ep,0}  = \F_{\ep}$, there exists a solution $U_{\ep,\mu}$,
which is analytic in the two parameters for arbitrarily small values of $\mu$.
\end{lemma}

\subsubsection{ Proof of Lemma~\ref{twovar} }

Given a family of operators depending on two parameters $\F_{\varepsilon, \mu}$, we fix
$\alpha$, $\beta$ and consider the one parameter family defined as
$\G_\varepsilon = \F_{\varepsilon, \alpha \varepsilon + \beta}$. By the
hypothesis, if $\alpha, \beta$ are small, we can find a solution $U_\ep$, so that
$\G_\ep(U_\ep) = 0$.

Geometrically, if we let $\beta$ vary, then the lines
$(\varepsilon,\alpha \varepsilon + \beta)$ form a foliation. For a
different value of $\alpha$, we obtain a transversal foliation.  The
solution is analytic when we restrict it to the leaves of each of the two
transversal foliations. Note that we are using the hypothesis of local uniqueness
to conclude that the solutions for two families are the same.

Precisely, if we choose $\alpha_1 \ne \alpha_2$, we can consider a change of coordinates
from $(\ep, \mu)$ to $\beta_1, \beta_2$  given by
\beq{trivialchange}
\alpha_1 \ep - \mu = \beta_1\ , \quad \quad \alpha_2 \ep - \mu = \beta_2\ ,
\eeq
which gives
$$
\ep = \frac{-\beta_1 + \beta_2}{-\alpha_1 + \alpha_2}\ , \quad \quad
\mu  = \frac{-\alpha_2 \beta_1 + \alpha_1  \beta_2}{-\alpha_1 + \alpha_2}\ .
$$
By hypothesis the solution $U_{\beta_1,\beta_2}$ of the two-parameter family $\F_{\ep,\mu}$
is analytic in $\beta_1$ for $\beta_2$ fixed and in $\beta_2$ for $\beta_1$ fixed.
This is the hypothesis of Hartogs theorem (\cite{Krantz,Narasimhan}),
so that we can conclude that the function $U_{\beta_1,\beta_2}$ is jointly analytic in $\beta_1$, $\beta_2$ and,
hence, it is jointly analytic in $\ep$, $\mu$.

If the operators act on infinite dimensional Banach spaces, we can reduce the proof to the classical
result for complex valued function by observing that we can  apply
Hartogs theorem to $\ell(U_\ep)$, where $\ell$ is a linear functional from the Banach space to the complex.
It is also well known (\cite{Reed-Simon-Vol1,HilleP}) that functions that are analytic in this weak sense are strongly
analytic.
Alternatively, we could just note that the proof of Hartog's theorem works for functions taking values in Banach spaces.
\qed

\begin{remark}\label{Poincare-Hartogs}
It is amusing to note that Lemma~\ref{twovar} allows one to improve
the results of \cite{Poincare92}. It immediately shows that if all systems in
a neighborhood remained analytically integrable, any two parameter family would be
uniformly integrable. Hence, the obstructions to uniform integrability discovered
by \cite{Poincare92} show that we can get non-integrable systems in any neighborhood.
Of course, even if Poincar\'e was one of the creators of the theory of several complex variables,
he did not know about Hartogs theorem.  Under the extra assumption of
uniform boundedness (which is not so unreasonable in the present case), the analogue of
Hartogs theorem was presumably known.
\end{remark}

Our next result shows that there are obstructions to the existence of
solutions analytic in two variables in two-parameter families near resonances.
This is an elementary application of power-series matching. Notice that the argument works in the generality of
mappings into Banach spaces, since it is really a soft
argument which applies in many other contexts.

One subtlety is that, if we consider the restricted class of families of operators as in \equ{equ:rearranged2}, it can, in principle,
happen that the obstructions vanish for the restricted family. So, when we consider restricted families such as
\eqref{equ:rearranged2}, we will need to verify that the family is general enough to be affected by the
obstructions.

\begin{lemma} \label{lem:obstruction}
Consider the two-parameter family $\F_{\ep, \mu}$.
For some $\ep_0$, assume that the following equation holds: $\F_{\ep_0,0}(U_{\ep_0,0})  = 0$.
If the range of $D\F_{\ep_0,0}$  has  codimension  at least $1$,
then the space of families for which there is a solution is
contained in a set of infinite codimension.

Moreover, for the restricted families of the form \eqref{equ:rearranged2}, if we can find
$U_{\ep, 0}$ solving $\F_{\ep, 0}(U_{\ep,0}) = 0$, then there exists an arbitrarily
small $\mu$, such that the family $\F_{\ep, \mu}$ does not have solutions
close to $U_{\ep,0}$, which are analytic near $\mu  = 0$.
\end{lemma}

{\bf Proof of Lemma~\ref{lem:obstruction}.}

If there is a solution $U_{\ep, \mu}$ of $\F_{\ep, \mu}(U_{\ep, \mu})=0$ analytic in $\mu$, we
should have:
\begin{equation}\label{firstobs}
D\F_{\ep_0,0}(U_{\ep_0,0})\ \partial_\mu U_{\ep_0,0}
+ \partial_\mu\ \F_{\ep_0,0}(U_{\ep_0,0}) = 0\ .
\end{equation}

Clearly, if the perturbation is such that
$ \partial_\mu \F_{\ep_0,0}(U_{\ep_0,0})$ is not in the
range of $D\F_{\ep_0,0}$, then there is no possibility of finding a solution
of \eqref{firstobs} and, a fortiori, no possibility of finding an
analytic solution.

Of course, the families $\F_{\ep, \mu}$ for which the first jet is
in the range is a codimension one set of perturbations. Hence,
the derived necessary conditions imply that the perturbation has to be
in this set.

Obviously, the necessary condition above is not the
only one. Indeed, one can obtain even more obstructions for the existence
of another branch  by considering higher order terms.
Matching terms up to order $N$, we obtain that
\begin{equation}\label{obstructionN}
D\F_{\ep_0,0}(U_{\ep_0,0})(\partial_\mu)^N U_{\ep_0,0}
+ (\partial_\mu)^N  \F_{\ep_0,0}(U_{\ep_0,0})  +
R_N = 0\ ,
\end{equation}
where $R_N$ is an expression involving only derivatives of order
up to $N-1$.

Clearly, the fact that $R_N+ (\partial_\mu)^N \F_{\ep_0,0}(U_{\ep_0,0}) $ is in the range of $D\F_{\ep_0,0}$ gives
another obstruction for the perturbations.

If the range of $D\F_{\ep_0,0}$ has codimension $k$, we claim that
the set of families that matches
the necessary conditions up to order $N$ is a submanifold of
codimension $Nk$.  In particular, the set of maps that satisfy
all the obstructions is contained in a submanifold of infinite codimension, which becomes
a very meager set in the sense of Baire category theory.

The proof of the first claim of the lemma is very easy. The key observation is that the
obstruction at order $N$ (see \eqref{obstructionN}) involves that a
given expression is in the range of $D\F_{\ep_0, 0}$. This expression is very
complicated in the coefficients of order up to $N -1$, but its
dependence on the coefficient of order $N$ is very simple. Hence, we
obtain that for each of the functions satisfying the condition at
order $N-1$, the obstruction at order $N$ takes the form that the
$N$--th derivative should be an explicit expression over all the previous ones
plus the range of $D\F_{\ep_0,0}$. Since the range of $D\F_{\ep_0,0}$
has codimension $k$, this increments by $k$ the codimension of the solution of
\equ{obstructionN}. In the limit we obtain that the solution of \equ{firstobs} is
contained in a set of infinite codimension.

The second claim of the lemma is obtained observing that for the families of operators
as in \eqref{equ:rearranged2}, then equation \equ{firstobs} gives restrictions
to the derivatives $\partial_\mu \F_{\ep_0,0}$. Then, we want to show
that the range of $\partial_\mu \F_{\ep_0,0}$ is in the complementary of the range of $D\F_{\ep_0,0}$.
Given that in the restricted family one has $\partial_\mu \F_{\ep,\mu}=\partial_\mu  A_{\mu}(U_{\ep,\mu})$
and recalling that the set of $A_{\mu}$ has infinite codimension, then we conclude that the range
of $\partial_\mu \F_{\ep_0,0}$ is not in the range of $D\F_{\ep_0,0}$ and therefore the family $\F_{\ep, \mu}$
does not admit solutions close to $U_{\ep,0}$, which are analytic near $\mu  = 0$.
\qed

\smallskip

Now we are in the position to finish the proof of Theorem~\ref{thm:optimal}. Consider
the family $\F_\ep$ of analytic operators and let $\ep_0$ be a resonant value.
Let $B \subset \complex $ be a ball around $\ep_0$. We assume that all the perturbations
of the family admit an analytic solution that goes across the ball $B$.
If indeed there were solutions for all perturbations, then using
Lemma~\ref{twovar} the family should be analytic in two variables. However,
near a resonance, which is contained in the ball $B$,
by Lemma~\ref{lem:obstruction} we obtain that there are many perturbations for which this is impossible.
Hence, we conclude that the assumption that there were solutions for all
perturbations analytic in the ball $B$ is false. This provides part $(a)$ of Theorem~\ref{thm:optimal}.

We conclude by mentioning that part $(b)$ of Theorem~\ref{thm:optimal}
is obtained from a straightforward implementation of the second statement of Lemma~\ref{lem:obstruction}.
\qed

\smallskip

Note that the proof of Theorem~\ref{thm:optimal} goes by
contradiction. We started by assuming that all the systems gave solutions
that were
analytic in a ball and we concluded that they were not, except in a
set of infinite codimension. This, of
course, contradicts  the hypothesis that for any perturbation,
there are analytic solutions
extending through a neighborhood of the resonant $\ep_0$
and we conclude that there is one family which does not extend.

Unfortunately, this does not allow us to conclude anything beyond
the fact that there are families which do not admit solutions that  extend through the resonance.
Once we conclude that the hypothesis fails, we cannot obtain any of the conclusions  that we obtained from
assuming its existence (and which we used to derive a contradiction).
In particular, the argument  does not allow us to conclude that the set of families which extend is
infinite codimension. The infinite codimension statement
was predicated on the fact that we had at least a solution.

If, indeed, we could show that the set of functions for which a
solution extends through a resonance is infinite codimension, we
could use Baire category theorem to show that there is a residual
set of families for which the analyticity domain does not include
any resonance.  Even if the argument above does not allow us
to conclude that rigorously, we formulate the following conjecture.

\begin{conjecture}\label{conj:optimal1}
For an open and dense set of families (in the topology indicated above),
there is no solution defined in a neighborhood of any of the resonances.
\end{conjecture}

Notice also that for the equations considered in \eqref{equ:rearranged2},
if we have a perturbative solution of the equation, the resonances of the perturbed equation have to be close to the
resonances of $N_\ep$. Hence, we also have the following conjecture.

\begin{conjecture}\label{conj:optimal2}
Consider the problem in \eqref{equ:rearranged2}. For an open and dense set of nonlinearities,
the response solution has a singularity at a distance less than $C|\ep|$ from the resonances of $N_\ep$.
\end{conjecture}

We hope that, perhaps, the argument used in Theorem~\ref{thm:optimal} can be strengthened
to obtain Conjecture~\ref{conj:optimal2}. There could also be
other strategies to prove Theorem~\ref{thm:optimal}, which are direct and not just by contradiction.
A more constructive argument could possibly take the form of observing that, near the resonances,
one small change in the model leads to a very large change in the response function. Hence, one could
hope to pile up perturbations of the model in such a way that the model remains well defined, but that the
response function breaks down.\\

To apply the above results to our models, one slightly delicate point is that the linearization depends on the solution
$U_\ep$. Arguing again by contradiction and in a non-rigorous way, we observe that we can compute the eigenvalues of
$N_\ep U_\ep+\mu  G(U_\ep)$ by using a perturbative expansion as in \cite{Kato}. Even if a full proof will be complicated,
one can imagine that the eigenvalues can be continued analytically in $\mu$. Hence, the values of $\ep$ for which an eigenvalue vanishes
will move continuously (of course, if there are some non-degeneracy assumptions, which is reasonable to conjecture hold generically) and
they will move differentiably.

Hence, we are lead to the following conjecture.

\begin{conjecture}\label{conj:optimal3}
For a generic family, we can find a constant $C$, such that
no ball of the form $\{ \ep :\ | \ep -\ep_0| \le C \ep_0^2 \}$, with $\ep_0$ a resonance for $N_\ep$, is completely contained
in the domain of analyticity of the response function.  In other words, for each of the balls as above, we can find
a point not in the domain of analyticity.
\end{conjecture}

\section*{Acknowledgements}
Part of this work was was carried out
while R.C. was a Visiting Assistant Professor
at the School of Mathematics in the Georgia Institute of Technology. R.C.
would also like to thank G. Flores for useful discussions
and for sharing a preprint
version of \cite{Flores13}.
A.C. is grateful to S. Terracini for several discussions.

\appendix

\section{Some properties of $\A_{\rho,j, m}$}\label{app.spaces}

\subsection{Characterization of the norm in terms of the Fourier coefficients}
Here we provide a norm equivalent to
\eqref{analyticnorm} which can be expressed in terms of the Fourier coefficients.  In  this way, it is
 easy to study the boundedness of operators which
are diagonal in the Fourier basis (products of
complex exponentials in $\theta$ and eigenfunctions of
$\Lop$ in $x$).
As before for two equivalent norms $\|\cdot\|$, $\|\cdot\|'$,
we will write $\|\cdot\|\cong\|\cdot\|'$.

\begin{proposition}
Let $u \in \A_{\rho,j,m}$ have a Fourier expansion as in
\eqref{generalizedFourier}. We have:
\beqano
\begin{split}
\|u\|_{\rho,j, m}  & \cong
\left( \sum_{k \in \integer^d\backslash\{0\}} \frac{e^{4 \pi |k| \rho}}{B(k,\rho)} ((2\pi)^d\ |k|^2 + 1)^{j} \,
\| \hat u_k\|_{H^m_\Lop}^2  +
{{\|\hat u_0\|_{H_\Lop^m}^2}\over {B(0,\rho)}}\right)^{1/2}\\
& =\left( \sum_{k \in \integer^d\backslash\{0\}, n \in \nat}
\frac{e^{4 \pi |k| \rho}}{B(k,\rho)}((2\pi)^d\ |k|^2 + 1)^{j} \lambda_n^m
 |\hat u_{k,n}|^2  +
{{\lambda_n^m}\over {B(0,\rho)}}|\hat u_{0,n}|^2\right)^{1/2}\ ,\\
\end{split}
\eeqano
where for $k \in \integer^d$ we denote $|k| \equiv |k_1| + \cdots +|k_d|$
and
\[
B(k,\rho) \equiv\prod_{j=1}^d a(k_j,\rho)\ ,\qquad
a(j,\rho) \equiv \left\{ \begin{array}{ll}
4 \pi |j| & \textrm{if }\, j\neq 0\\
{1 \over {4\pi \rho}} & \textrm{if }\, j = 0\\
\end{array}\right. .
\]
\end{proposition}

\begin{proof}
For $d=1$ we have that for $k \ne 0$:
$$
\int_{\torus_\rho} | e^{2 \pi i k\theta}|^2 d^2 \theta =
\int_{|\Im(\theta)|\leq \rho} \ e^{-4 \pi k \Im \theta}\ d (\Im(\theta))
= \frac{e^{4 \pi |k| \rho}-e^{-4 \pi |k| \rho}}{ a(k,\rho)}\ .
$$
Of course, when $k = 0$, the integral is just $4\pi\rho$.
For any $k\neq0$ we have that the integral can be bounded as
\beq{bound}
C_-\ \frac{e^{4 \pi |k| \rho}}{ a(k,\rho)}\leq \int_{\torus_\rho} | e^{2 \pi i k\theta}|^2 d^2 \theta
\leq C_+\ \frac{e^{4 \pi |k| \rho}}{ a(k,\rho)}
\eeq
with $C_-\equiv 1-e^{-4\pi\rho}$ and $C_+\equiv 2$. For $d\geq 2$ we can use Fubini's theorem, applying \equ{bound}
to each factor and obtaining the following inequalities:
$$
C_-^d\ {e^{4 \pi |k| \rho}\over {B(k,\rho)}}\leq \int_{\torus^d_\rho} | e^{2 \pi i k\cdot \theta}|^2 d^{2d} \theta
\leq C_+^d\ {e^{4 \pi |k| \rho}\over {B(k,\rho)}}\ .
$$
Finally we note that the  Laplacian is diagonal in the exponentials, so that
\[
|(\Delta_\theta + 1)^{j/2} e^{2 \pi i k\cdot \theta}|^2
= ((2 \pi)^d \, |k|_2^2 + 1)^j  |e^{2 \pi i k\cdot \theta}|^2\ ,
\]
where $|k|_2$ denotes the Euclidean norm.

Since the exponentials are orthogonal with respect to the $L^2$ inner product, we obtain
\[
\int_{\torus^d_\rho} \|(\Delta_\theta + 1)^{j\over 2}  u(\theta, \cdot)\|_{H_\Lop^m}^2\ d^{2d} \theta
\cong \sum_{k\in\integer^d\backslash\{0\}} \frac{e^{4 \pi |k| \rho}}{B(k,\rho)} ((2 \pi)^d \, |k|_2^2 + 1)^j  \|\hat u_k\|_{H_\Lop^m}^2+
{{\|\hat u_0\|_{H_\Lop^m}^2}\over {B(0,\rho)}}\ .
\]
This concludes the proof.
\end{proof}

\begin{remark}
Since the Euclidean norm $|k|_2$ in $d$ dimensions is equivalent to the $\ell^1$-norm $|k|$,
 we can substitute  $|k|_2$ for $|k|$ in the polynomial factor
and  obtain an equivalent norm in the Banach space. On the other hand, if
we change the $\ell^1$-norm of $|k|$
in the argument of the exponential by the $\ell^2$-norm, we obtain
a non-equivalent norm in the space of analytic functions.
\end{remark}

The space  $\A_{\rho, j, m}$ is a closed subspace of the
Sobolev space of maps from the $2d$
dimensional manifold $\torus_\rho^d$
into the Banach algebra $H_\Lop^m$.
Indeed for $m>\ell/2$ (i.e., in the case considered in this paper), we can (i) apply Sobolev embedding theorem,
(ii) obtain Banach algebra properties under multiplication,
(iii) apply Gagliardo-Nirenberg-Moser inequalities for composition.

In particular, since $\torus^d_\rho$ is a $2d$-dimensional real manifold, the Sobolev embedding theorem implies
\[\|u\|_{L^{\infty}(\torus^d_\rho; H_\Lop^m)} \leq C \|u\|_{H^{j}(\torus^d_\rho;H_\Lop^m)}\ ,\]
whenever $j>d$ for some constant $C>0$. Hence, the convergence in $H^{j}(\torus_\rho^d ; H_\Lop^m)$ implies the
uniform convergence and it is well known that uniform limits
of analytic functions are analytic.

Summarizing we have the following properties.
\begin{proposition}\label{properties}
Let $\rho>0$, $j,m \in 2 \nat$.
The space $\A_{\rho, j, m}$ of functions analytic in $\theta$, endowed with the norm given in
\equ{analyticnorm} is a Banach algebra
under multiplication, when $j > d$, $m > \ell/2$.
\end{proposition}

\begin{proposition}\label{multiplier}
If we have a linear  operator $\M$  which is diagonal in the Fourier basis, say
$$
\M(e^{ 2 \pi i k \cdot \theta} \Phi_n(x)) =
\lambda_{n,k}\ e^{ 2 \pi i k \cdot \theta} \Phi_n(x)\ ,
$$
for suitable coefficients $\lambda_{n,k}$, then we have
\beq{eqnorm}
\| \M \|_{\rho, j, m} \le C \sup_{n,k} |\lambda_{n,k} |\ .
\eeq
\end{proposition}

The bound \eqref{eqnorm} is immediate since the operator $\M$ is diagonal in the Fourier
basis and the norm is just the sum of the Fourier coefficients.

\begin{remark}\label{rem:cont}
We can also dispense with the assumption that $\Lop$ has a
discrete spectrum, provided we assume it is self-adjoint.

In general, if $\Lop$ is self-adjoint in $L^2_{BC}(\D)$, the spectral
theorem for self-adjoint operators in separable
Hilbert spaces (\cite{Dunford, Helmberg, Reed-Simon-Vol1, VonNeumann})
shows that there is a positive  measure $\mu$ defined on the reals
and a unitary operator $V$
from $L^2_{BC} (\D)$ to $L^2(M, \mu)$, where
$M$ is a measure space, such that
\[
[V \Lop V^{-1}  f] (x)  = F(x)  f(x)
\]
for some multiplication operator $F$. We recall that the role of $M$ is to account for the multiplicity.  Very often $M$ is just
copies of $\real$, the number of copies being the maximum multiplicity. The multiplication operator, then, is just
a multiplication by the coordinate $x$.
Of course, the unitary mapping and the measure $\mu$ depend heavily on
$\Lop$, even if we do not include it in the notation.

The above relation means that the unitary transformation sends the operator $\Lop$ into
a multiplication by $F(x)$. This allows us to define the generalized Sobolev norm as
\[
|| f ||_{H^m, \Lop}^2 = \int_\real F(x)^{2m} |Vf(x)|^2 \, d \mu\ .
\]

This is quite analogous to the characterization of Sobolev
spaces in terms of the Fourier series developed in \cite{TaylorII}.
Again, by G\"arding's inequality, these spaces are equivalent to
the standard Sobolev spaces in the variable $x$ and we can define
$\A_{\rho, j, m}$ as the spaces of analytic functions from
$\torus^d_\rho$ to $H^m_\Lop$.

We can also lift the spectral theorem to Fourier series.
Indeed, we can consider the space
\[
\ell^2_{BC} (\torus^d \times \D) = L^2(\torus^d) \otimes  L^2_{BC}(\D)\ .
\]
Note that the Fourier transform $\F$ (if appropriately normalized) is
a unitary operator from $L^2(\torus^d)$ to
$\ell^2(\mathbb{Z})  =  L^2( \real, \sum_n \delta_n)$.
Hence we can consider the operator
\[
\tilde V  \equiv \F \otimes V:  L^2(\torus^d) \otimes  L^2_{BC}(\D)
\rightarrow
L^2(\real \times \real, \Big( \sum_n \delta_n\Big)  \otimes d\mu)\ .
\]
The analogue of Proposition~\ref{multiplier} is given by the following result.

\begin{proposition} \label{multiplier2}
If we have a linear  operator $\M$  such that
$$
(\tilde V \M \tilde V^{-1} f)(k, x) = \lambda(k,x)  f(k, x)\ , $$
then, we obtain
\def\supp{ {\rm support}}
\[
\| \M \|_{\rho, j, m} \le C \sup_{k \in \integer, x \in \supp(d \mu)} | \lambda(k,x)|
\]
for a positive constant $C$.
\end{proposition}

It is important to note that the assumptions that we used in
Section~\ref{sec:bounds},
giving uniform bounds on the operator
$N_\varepsilon$, only assumed that the spectrum lied in the real line
and they were uniform for all $\lambda$ real.

The main applications of the case of non-discrete spectrum appear when $\D$ is an unbounded domain of
$\real^d$ or another unbounded  manifold. In that case, one has to take care of
the fact that the Sobolev embeddings are different and the global
regularity theory for elliptic equations may be different.
\end{remark}

\subsection{Analytic functions from a Banach space to another}
\label{sec:analytic}

Now we present some analyticity properties of nonlinear functions in
Banach spaces (the literature on this subject is very wide, see for
example \cite{HilleP, Mujica}).  In analogy with the finite dimensional
case, there are results which show that some weak definitions
such as differentiability  imply stronger ones (convergence of
Taylor series around every point). In the infinite dimensional cases
the results are more subtle since there are
different notions of differentiability and different notions of
convergence of power series, but it is true that extremely weak notions
turn out to be equivalent to the strongest one (\cite[Chapter III]{HilleP}).

For our purposes, we only need to apply an easy implicit function theorem and to
study the analyticity properties of the nonlinear operator $G$ defined in
\eqref{nonlinearG}. The Banach spaces in which $G$ acts are
the spaces defined in Section~\ref{sec:spaces-findim-Banach}.

The following definition of analytic functions will be enough for us.

\begin{definition}\label{def:analytic}
Let $X$ and $Y$ be complex Banach spaces.
We say that $f:\Omega \subset X \to Y$ is analytic if it is
uniformly differentiable
at all points of $\Omega$, namely the derivative is uniformly bounded and
there exists a function $\gamma=\gamma(|z|)$, with $\gamma(|z|)\to0$ as $|z|\to 0$,
such that the following uniform bound holds:
\begin{equation}\label{Gateaux}
\|f(x + z \zeta) - f(x) - z Df(x)\zeta\| \leq \gamma( |z|)\
\end{equation}
for all $x \in \Omega$, $\zeta \in X$ with $|\zeta| = 1$.
\end{definition}

\begin{remark}
It is clear that Definition~\ref{def:analytic}   is a rather weak notion of
differentiability. However, it is a remarkable
fact (see \cite{HilleP}, Theorem 3.17.1) that Definition~\ref{def:analytic}
is equivalent to requiring that the function $f$ has a Taylor expansion of the form
$$
f(x+\zeta) = \sum_{\alpha} {1 \over \alpha!}\ \partial^\alpha f(x) \zeta^\alpha\ ,
$$
that converges uniformly for $\| \zeta\|\leq M$ for some $M>0$
(indeed, $\sum_\alpha \| {1 \over \alpha!} \partial^\alpha f(x)\|M^{|\alpha|} < \infty$).

It is also true that even weaker
notions of differentiability imply \eqref{Gateaux}. For
our purposes, Definition~\ref{def:analytic} will be enough, since it allows us
to apply a contraction mapping argument.
We refer to \cite[Chapter III]{HilleP}, \cite{Mujica,Herve} for
other definitions of analyticity that turn out to be equivalent.
\end{remark}

To make sense of the composition, we need an analytic extension of the non-linearity $h$ to the complex plane
with respect to its first argument.

\begin{proposition}\label{prop:composition}
Let $h:\mathcal{B} \times \overline{\D} \to \complex$ with $\mathcal{B}$ an open set
in $\complex$ and assume that $h$ is analytic in  $u \in \mathcal{B}$ and it is
$C^m(\D)\cap C(\overline{\D})$.

Assume that $\partial_u ^{\alpha_1}  \partial_x^{\alpha_2} h(u, x)$
are bounded in $\mathcal{B} \times \D$
for $\alpha_1 + \alpha_2 \le m$.
We denote by $\partial_u$ the complex derivatives
and we assume that $m > \ell/2$.

Let $u_0 \in \A_{\rho, j, m}$ be such that $\dist(u_0(\torus^d_\rho, \overline{\D}), \complex  \setminus \mathcal{B})>a_0$
with $a_0>0$.

Then, for all $u$ in a neighborhood $\mathcal{U}$ of $u_0$ in
$\A_{\rho, j, m}$, we can define the operator
$$
\C_h[u](\theta,x) = h(u(\theta,x),x)
$$
from $\A_{\rho, j, m}$ to itself, which is an analytic operator in the sense of
Definition~\ref{def:analytic}.

Moreover, for $v\in\A_{\rho, j, m}$ the derivative of $\C_h$
is given by
$$
(D\C_h[u] v)(\theta,x) =
h'( u(\theta,x), x)  v(\theta,x)\ ,
$$
where by $h'$ we denote the complex derivative of $h$
with respect to its first argument, namely $h'(u, x) = (\partial_1 h)( u, x)$ (in the proof $h''$ will denote the second
derivative of $h$ with respect to its first argument).
\end{proposition}

\begin{proof}
The proof is rather straightforward, but it requires  that
we interpret some elementary calculations (the Taylor
theorem up to order two with remainder) in different levels
of abstraction.

Because of Sobolev's embedding theorem we have that the functions
$u, v$ are bounded and continuous, so that, for fixed
$\theta, x$ and for fixed $u,v$, we can think of
$u(\theta, x), v(\theta, x)$ as numbers
and assume that $|v(\theta, x)|$ is so small that
$u(\theta, x) + s\, v(\theta, x)$ is in the domain of $h$ for $s\in[0,1]$.

The fundamental theorem of calculus implies that
for all $\theta \in \torus_\rho^d ,x \in \D$ and for all
$u \in \mathcal{U}$ and all $v \in \A_{\rho, j, m}$, we have:
\begin{equation}\label{Taylor}
\begin{split}
h(u(\theta, x) + v(\theta, x), x) =&\ h(u(\theta, x), x) + \int_0^1 h'(u(\theta, x)
+ s\,v(\theta, x), x) v(\theta, x)\ ds\\
=&\ h(u(\theta, x), x) + h'(u(\theta, x),x) v(\theta, x) \\& +
\int_0^1 \int_0^s h''(u(\theta, x) + s\  t\ v(\theta, x), x)\ v(\theta, x)^2\ \ dt\ ds\ .\\
\end{split}
\end{equation}

Now we interpret the formula \eqref{Taylor} as an equality in function spaces.

By standard Gagliardo-Nirenberg-Moser composition estimates in Sobolev
spaces (\cite[Proposition 3.9]{TaylorIII}),
we have that for some $C_\alpha>0$ depending on the norm of $u$, one has:
\[\|D^\alpha h(u,x)\|_{H^j(\torus_\rho^d ; H_\Lop^m)} \leq C_{\alpha}(\|u\|_{L^\infty(\torus^d_\rho;\D)})\
(1+\|u\|_{H^j(\torus_\rho^d ; H_\Lop^m)})\ .\]

It is also easy to check that if $u(\theta,x), v(\theta, x) $ are
 complex differentiable
in $\theta$ and  $h$ is also differentiable in $\theta$,
we obtain that $ h''(u(\theta, x) + s\  t\ v(\theta, x), x)$ is
a function in $\A_{\rho,j, m}$ with uniform bounds.

Using the Banach algebra properties, we obtain indeed the desired result,
since we can bound the integral by a constant times $\|v\|_{\rho, j, m}^2 $ in the last term of \equ{Taylor}.
\end{proof}

\section{Solution of the zeroth order term for models A and A'}\label{sec:c0A}

For models A, A' the zeroth order term $c_0$ must satisfy \equ{c0}. The literature on the solution
of \equ{c0} is very wide and the results strongly depend on the form of the non-linearity. To be
concrete, we quote - among the others - a result on the existence of weak solutions (see Proposition~\ref{pro:weak}),
which are indeed regular solutions as noticed in Remark~\ref{rem:weak} below, and a result on the existence of an unbounded sequence
of solutions if $h$ is odd (see Proposition~\ref{pro:struwe}).
To fix the notation, let $f_0(x)\equiv \langle f(\theta,x)\rangle$.

\begin{proposition}(\cite{precup}, Theorem 9.7)\label{pro:weak}
Let $\D\subset \real^\ell$, $\ell\geq 3$, be a bounded open set, $f\in H^{-1}(\D)$;
let $\underline u$, $\overline u$ be, respectively, a lower and an upper
\footnote{By a lower (upper) solution of problem \equ{c0}, we mean a function
$\underline{u}\,(\overline{u}) \in H^1(\D)$, such that
$h({\underline u},\cdot)\, (h({\overline u},\cdot))
\in L^{\frac{2\ell}{\ell+2}}(\D)$ satisfies
\begin{equation*}
\begin{split}
-\Delta_x \underline u(x)+h(\underline u(x),x)&\leq \langle f\rangle(x)\\
(-\Delta_x \overline u(x)+h(\overline u(x),x)&\geq \langle f\rangle(x))\\
\end{split}
\end{equation*}
and $\underline u(x)\leq 0$ ($\overline u(x)\geq 0$) when $x\in \partial{\D}$.}
solution of \equ{c0}
with D boundary conditions and with ${\underline u}(x)\leq {\overline u}(x)$ for almost every $x\in \D$.
If the function $h:\D\times\real\rightarrow\real$ satisfies the Carath\'eodory
conditions\footnote{A function $h:\D\times\real\rightarrow\real$ satisfies the Carath\'eodory conditions,
if $h(y,\cdot):\D\rightarrow\real$ is measurable for every $y\in\real$ and $h(\cdot,x):\real\rightarrow\real$
is continuous for almost every $x\in\D$.} and is increasing in its first variable, i.e.
$$
h(u_1,x)\leq h(u_2,x)
$$
with ${\underline u}(x)\leq u_1\leq u_2\leq{\overline u}(x)$ for almost every $x\in\D$, then
\equ{c0} with D boundary conditions has at least one weak solution $u\in H_0^1(\D)$,
satisfying ${\underline u}(x)\leq u(x)\leq{\overline u}(x)$ for almost every $x\in\D$.
\end{proposition}

\begin{remark}\label{rem:weak}
For the models we are considering, it can be proved that all the weak solutions are smooth:
under regularity conditions on the coefficients defining the elliptic operator $\Delta_x$, then a weak solution
is regular.
We refer the reader to Chapter 4 in \cite{Ladyzhenskaya} and to \cite{Agmon, evans}.
\end{remark}

\vskip.1in

Assuming that $h$ is odd, one obtains an infinite number of solutions as shown by the following result.

\begin{proposition}(\cite{struwe}, Theorem 7.2)\label{pro:struwe}
Let $\D\subset \real^\ell$, $\ell\geq 3$, be a smoothly bounded domain. Assume that

$i)$ $h:\D\times\real\rightarrow\real$ is continuous and odd with primitive $h_p(x,u)=\int_0^u h(x,v) dv$;

$ii)$ there exists $p<2\ell/(\ell-2)$ and $C>0$, such that
    $$
    |h(x,u)|\leq C(1+|u|^{p-1})
    $$
    almost everywhere;

$iii)$ there exists $q>2$, $R_0>0$ such that
    $$
    0<q\ h_p(x,v)\leq h(x,v)\ v
    $$
    for almost every $x$, $|v|\geq R_0$;

$iv)$ the quantities $p$ and $q$ satisfy
    $$
    {{2p}\over {\ell(p-2)}}-1>{q\over {q-1}}\ .
    $$
Then, for any $f_0\in L^2(\D)$, the equation \equ{c0} with $c_0(x)=0$ on the boundary of $\D$,
has an unbounded sequence of solutions $c_{0k}\in H_0^{1,2}(\D)$, $k\in\nat$.
\end{proposition}

\vskip.1in

\begin{remark}\label{rem:c0}
$i)$ In the above results, the Laplacian can typically be replaced by any second order uniformly
elliptic operator with smooth coefficients (compare with \cite{amb-malc}).

$ii)$ The multiplicity of solutions can be studied using Lusternik-Schnirelman theory, which allows one
to find critical points of the variational functional on a given manifold. The number of critical points
is related to the Lusternik-Schnirelman category of the manifold, for which a lower bound is provided
by the cup length of the manifold (\cite{amb-malc}).

$iii)$ In our discussion we considered D boundary conditions; results are available also for N
boundary conditions (see, e.g., \cite{Tang-Wu, Tang, Wu-Tan}) or can be extended to
P boundary conditions.
\end{remark}

\bibliographystyle{alpha}
\bibliography{cccl1}

\end{document}